\documentclass[11pt, a4paper]{amsart}

    \usepackage[T1]{fontenc}
    \usepackage[utf8]{inputenc}
    \usepackage{lmodern}

    \usepackage[american]{babel}
    \usepackage[babel, final]{microtype}
    \usepackage{amssymb,mathrsfs,bbold}
    \usepackage{amsthm}
    \usepackage{enumitem}
    \usepackage{afterpage}
    
\makeatletter
\DeclareFontFamily{U}{mathx}{\hyphenchar\font45}
\DeclareFontShape{U}{mathx}{m}{n}{
      <5> <6> <7> <8> <9> <10>
      <10.95> <12> <14.4> <17.28> <20.74> <24.88>
      mathx10
      }{}
\DeclareSymbolFont{mathx}{U}{mathx}{m}{n}
\DeclareFontSubstitution{U}{mathx}{m}{n}

\let\widecheck\@undefined
\let\widebar\@undefined
\DeclareMathAccent{\widecheck}{\mathord}{mathx}{"71}
\DeclareMathAccent{\widebar}{\mathord}{mathx}{"73}
\makeatother

\makeatletter
\DeclareSymbolFont{stmry}{U}{stmry}{m}{n}
\SetSymbolFont{stmry}{bold}{U}{stmry}{b}{n}

\let\llbracket\@undefined
\let\rrbracket\@undefined
\DeclareMathDelimiter{\llbracket}{\mathopen}%
                     {stmry}{"4A}{stmry}{"71}
\DeclareMathDelimiter{\rrbracket}{\mathclose}%
                     {stmry}{"4B}{stmry}{"79}
\makeatother

    \usepackage{mathtools}
    \usepackage{ifpdf}
    \usepackage{comment}
    \usepackage{multirow}
    \usepackage[pdftex, dvipsnames]{xcolor}
    \usepackage[pdftex, final]{graphicx}
    \usepackage{caption}
    \usepackage{subcaption}
    \usepackage[pdftex,%
        a4paper,
        includehead,%
        includefoot,%
        nomarginpar,%
        lmargin=1in,
        rmargin=1in,%
        tmargin=1in,%
        bmargin=1in,%
    ]{geometry}
    \usepackage[pdftex,%
        final,%
        colorlinks=true,%
        linkcolor=NavyBlue,%
        citecolor=NavyBlue,%
        filecolor=NavyBlue,%
        menucolor=NavyBlue,%
        urlcolor=NavyBlue,%
        bookmarks=true,%
        bookmarksdepth=3,%
        bookmarksnumbered=true,%
        bookmarksopen=true,%
        bookmarksopenlevel=2,%
    ]{hyperref}
    \hypersetup{
        pdftitle={The Complexity of Shake Slice Knots},
        pdfauthor={Charles Ransome Stine}          
    }
    \usepackage{cleveref}
    \usepackage{multicol}
    \usepackage{tikz}
    \usetikzlibrary{tikzmark}
    \usepackage{pinlabel}
    \usepackage{todonotes}
    \usepackage{marginnote}

\input{macros}

\newcommand{\invs}{^{-1}}
\newcommand{\half}{\frac{1}{2}}

\newcommand{\set}[1]{\mathchoice%
  {\left\lbrace #1 \right\rbrace}%
  {\lbrace #1 \rbrace}%
  {\lbrace #1 \rbrace}%
  {\lbrace #1 \rbrace}%
}

\newcommand{\abs}[1]{\mathchoice%
  {\left\lvert #1 \right\rvert}%
  {\lvert #1 \rvert}%
  {\lvert #1 \rvert}%
  {\lvert #1 \rvert}%
}

\renewcommand{\bar}{\overline}
\renewcommand{\hat}{\widehat}
\newcommand{\dual}{^*}
\newcommand{\pound}{_{\#}}
\newcommand{\ppp}{^{\prime\prime\prime}}
\newcommand{\pp}{^{\prime\prime}}        

\newcommand{\of}{\circ}                    

\newcommand{\into}{\hookrightarrow}

\newcommand{\numset}[1]{\mathbb{#1}}
\newcommand{\one}{\mathbb{1}}
\newcommand{\N}{{\numset{Z}_{>0}}}
\newcommand{\Z}{\numset{Z}}
\newcommand{\Zpos}{\Z_{\geq 0}}
\newcommand{\Q}{\numset{Q}}
\newcommand{\R}{\numset{R}}
\newcommand{\C}{\numset{C}}

\DeclareMathOperator{\Id}{Id}

\DeclareMathOperator{\double}{double}

\newcommand{\diffeo}{\cong}

\newcommand{\boundary}{\partial} 
\newcommand{\connsum}{\mathbin{\#}} 

\newcommand{\nbhd}[1]{\nu (#1)}

\DeclareMathOperator{\cocore}{cocore}
\DeclareMathOperator{\sign}{sign}

\newcommand{\sphere}[1]{{S}^{#1}}
\newcommand{\disk}[1]{{D}^{#1}}
\newcommand{\ball}[1]{{B}^{#1}}

\renewcommand{\SS}[2]{{S}^{#1}\times{S}^{#2}}

\newcommand{\DS}[2]{{D}^{#1}\times{S}^{#2}}
\newcommand{\DD}[2]{{D}^{#1}\times{D}^{#2}}

    \title{The Complexity of Shake Slice Knots}                                              
    \author{Charles Ransome Stine}
    \address{Brandeis University}
    \email{\href{mailto:crt64@brandeis.edu}{crt64@brandeis.edu}}
    \urladdr{https://sites.google.com/view/charlesrstine}

    \def\subjclassname{\textup{2020} Mathematics Subject Classification}
    \expandafter\let\csname subjclassname@1991\endcsname=\subjclassname
    \expandafter\let\csname subjclassname@2000\endcsname=\subjclassname
    \subjclass{57K40 
        (Primary),
        57K10 
        (Secondary);
        \hfill
        Date: \today
    } 

    \newcommand{\red}[1]{\textcolor{BrickRed}{#1}}
    \newcommand{\green}[1]{\textcolor{ForestGreen}{#1}}
    \newcommand{\blue}[1]{\textcolor{Blue}{#1}}
    \newcommand{\purple}[1]{\textcolor{RoyalPurple}{#1}}
    \newcommand{\orange}[1]{\textcolor{orange}{#1}}
    
    \newcommand{\sI}{{\mathscr I}}
    \newcommand{\sG}{{\mathscr G}}
    \DeclareMathOperator{\Int}{Int}

    \DeclareMathOperator{\wrap}{wrap}
    \newcommand{\sub}[1]{_{#1}}

\begin{document}

\begin{abstract}     
    We define a notion of complexity for shake-slice knots which is analogous to the definition of complexity for h-cobordisms studied by Morgan-Szabó. We prove that for each framing $n \ne 0$ and complexity $c \ge 0$, there is an $n$-shake-slice knot with complexity at least $c$. Our construction makes use of dualizable patterns, and we include a crash course in their properties. We bound complexity by studying the behavior of the classical knot signature and the Levine-Tristram signature of a knot under the operation of twisting algebraically-one strands.
\end{abstract}

\maketitle

\section{Introduction}

    \begin{question}
        \label{qun:motivation}
        How far can a shake-slice knot be from being slice?
    \end{question}

    Given an oriented knot $K$ in $\sphere3$, one can construct the associated knot trace $X_n(K)$ by attaching an $n$-framed 2-handle to $\ball4$ along $K$ for each $n \in \Z$. Observe that if $K$ is a slice knot, then the generating class of $H_2(X_n(K),\Z)$ is represented by a smoothly embedded sphere, composed of two hemispheres: the slice disk for $K$ in $\ball4$ and the core of the attached 2-handle. It is natural to ask if this is the only way for $H_2$ of a knot trace to be generated by a smoothly embedded sphere. Akbulut answered this question negatively in \cite{akbulut_2-dimensional_1977}: he showed that there were non-slice knots whose $\pm1$ and $\pm 2$-traces contained smoothly embedded spheres generating $H_2$. In his subsequent paper \cite{akbulut_knots_1993}, he gave similar examples for each non-zero integer framing.

    Akbulut called these knots \emph{shake-slice}, and they have been studied in both the smooth and topological category, depending on whether the sphere generating $H_2$ of the trace is smoothly or topologically, locally-flatly embedded. Smoothly shake-slice knots were studied in \cite{cochran_shake_2016} in which the authors gave a characterization in terms of ribbon satellite patterns. Topologically shake-slice knots were characterized in turn by \cite{feller_embedding_2021}, who were able to answer \autoref{qun:motivation} for the framings $\pm 1$ in the topological category. 

    This paper considers two ways to measure how far a shake-slice knot (in either category) is from being slice: the first is simply the 4-genus of the knot, and the second is called the \emph{complexity}, which we denote $c^{TOP}$ or $c^\infty$. Like the genus, the complexity has the property that $c^\Box(K) \ge 0$ with equality if and only if $K$ is slice in the appropriate category. For each fixed $n \in \Z \setminus \set{0}$, $c \in \N$, we produce a smoothly $n$-shake-slice knot with 4-genus exactly $c$ (in both categories). We show that the complexity of the $n$-shake-slice knot is bounded below by its 4-genus (in each category respectively) and then we obtain an upper bound its smooth complexity by an explicit construction. We prove:

    \begin{theorem}
    \label{thm:main}
        Let $n \in \Z \setminus \set 0$ and $c \in \Z_{>0}$ be given. There exists a knot $K_{n,c} \subset \sphere3$ such that $K_{n,c}$ is smoothly $n$-shake-slice with 4-genus and complexity:
        $$ 
            g_4^{TOP}(K_{n,c}) = g_4^\infty(K_{n,c}) = c \le c^{TOP}(K_{n,c}) \le c^\infty(K_{n,c}) \le \half\left(3^c-1\right)
        $$
    \end{theorem} 

    The knots $K_{n,c}$ are satellite knots constructed using the dualizable-patterns technology of \cite{miller_knot_2018} and others. Our examples are somewhat reminiscent of the examples of \cite[Prop. 7.6]{cochran_shake_2016}, however Cochran-Ray assume the smooth 4D Poincare conjecture to prove their knots are smoothly shake-slice. Our examples, by contrast, are \emph{a priori} smoothly shake-slice because they are built out of patterns with smoothly slice duals. This technology has been in the literature for a while, but there is no single, complete reference for it. We give one here, and iron out a few kinks in the notation to produce a complete, symbolic calculus for dualizable patterns. We append a one page cheat sheet (\autopageref{cheatsheet}) with a concise summary of the calculus which may be used independently from the rest of this paper. 

    The main technical difficulty in this work is that very few invariants are well suited to compute the 4-genus of a smoothly $n$-shake-slice knot. Most invariants derived from gauge theory work by embedding some trace of the knot into a 4-manifold with nice geometric properties and then bounding the genus of the associated, primitive $H_2$ class. Such invariants vanish for smoothly shake-slice knots by definition. The Thurston-Bennequin inequality for Legendrian representatives, used in \cite{cochran_shake_2016}, does not work for our examples for this reason. The trace-invariance of the Heegaard Floer concordance invariants $\nu, \nu^-$ proved in \cite{hayden_exotic_2019} rules them out as well, and, by a small additional argument, rules out $\tau$. Rasmussen's $s$-invariant remains a possibility, but its behavior under satellite operations is too poorly understood for it to be practically computable. The $d$-invariants of finite, cyclic branched-covers do not obviously vanish for smoothly shake-slice knots, but they cannot be used to show the 4-genus is greater than one. Altogether, this rules out the modern tools of which we are aware.

    We use the topological 4-genus bound coming from the Levine-Tristram signatures, first introduced in \cite{tristram_cobordism_1969}, associated to a knot an a unit-norm complex number. As Conway points out in his excellent survey \cite{conway_levine-tristram_2021}, the LT-signatures of a knot can be computed either from a cyclic branched cover of $\ball4$ along a pushed-in Seifert surface, or from the twisted 2-homology of the 0-surgery. The main thrust of our argument plays these two definitions against each other to reduce the calculation for our knots $K_{n,c}$ to an elementary observation about the zeroes of an explicit family of Laurent polynomials on the unit circle in the $\C$. We see from the second definition that the LT-signatures are invariants of the 0-trace of $K$, which is why the framing $n=0$ is excluded from our theorem, but we learn that for all other framings the LT-signatures can provide arbitrarily high genus bounds for smoothly shake-slice knots. 

    Morally, this seems to indicate that 4-genus bounds arising from branched covers are well-suited to studying shake-slice knots. If some Floer-theoretic analogue of the LT-signature could be defined, stronger than the $d$-invariant, then it would likely be possible to construct smoothly shake-slice knots whose topological and smooth complexities differ. The author has constructed several families of examples which seem to have this behavior, but it is impossible to verify in the absence of such an invariant. We also remark that a satellite formula for the $s$-invariant of wrapping number three patterns would also likely suffice to verify these examples. 

    As an aside to the main result of this paper, we give a reformulation of the Goeritz-Trotter method for computing the classical signature of a knot, which, if not algorithmically quicker than procedure as described in \cite{litherland_signatures_1979}, is easier to perform as a human looking at a diagram of a knot. Using this reformulation, we are able to give a different, somewhat simpler proof of the following theorem due to Tristram:

    \begin{theorem}[\cite{tristram_cobordism_1969}, Cor. to Thm 3.2]
    \label{thm:2twists}
        Let $K$ be a knot in $\sphere3$, and let $U$ be an unknot disjoint from $K$ with $lk(K,U) = 1$. Let $K_n$ be the result of applying $n$ full twists to the strands of $K$ which intersect a spanning disk for $U$. For all $n \in \Z$, 
        $$
            \sigma(K_n) = \sigma(K_{n+2}) \quad\text{and thus}\quad \abs{\set{\sigma(K_n) \mid n \in \Z}} \le 2
        $$
    \end{theorem}

    This is in stark contrast to the behavior of the signature when we change the linking condition to $lk(K,U) = 0$. Under that condition the signature of $K_n$ decreases steadily until it hits a fixed, stable value and the number of times the signature decreases is roughly proportional to the number of algebraically cancelling pairs of strands which pass through $U$. When $lk(K,U) = \pm 1$, the signature oscillates between two fixed values, a behavior which is independent of the number of algebraically cancelling pairs of strands being twisted. Tristram's original proof proceeds by constructing a family of very large Seifert forms for the knots $K_n$ and then analyzing the behavior of the associated Hermitian forms. Our proof uses a much simpler family of non-orientable spanning surfaces for the knots $K_n$, which leads to simpler algebra.

    \subsection{Organization}

        We start by reviewing the definition and properties of dualizable patterns. We give two different diagrammatic methods for computing the dual of a pattern. We give a universal, combinatorial method for constructing dualizable patterns, and a systematic notation for building and manipulating satellite knots with dualizable patterns. We review the relevant properties of shake-slice knots in \S3, and give a diverse family of dualizable patterns. We prove that 4-genus bounds complexity, and we give sufficient conditions for a pattern and 4-genus lower bound to produce and detect high-complexity shake-slice knots. We digress in \S4 to give our reformulation of Goeritz and Trotter's method for computing the signature. We use it to prove \autoref{thm:2twists} and then use that to deduce a proof of \autoref{thm:main} for $n$ odd. Finally, in \S5, we verify the general sufficient conditions given in \S3 for a specific pattern and the Levine-Tristram signature function. This completes the proof of \autoref{thm:main} for both parities of $n$. 

    \subsection{Acknowledgements}

        This project is deeply indebted to Danny Ruberman, who suggested the question and provided a great deal of support along the way, and Lisa Piccirillo for many helpful conversations. The author thanks Joshua Wang for helping to fix an issue with the proof of \autoref{thm:2twists}, and Allison Miller for her comments on an early draft. The author thanks Arumina Ray for several helpful conversations regarding the results in \cite{cochran_shake_2016}, and encouraging him to continue sharpening the result. Finally the author thanks the anonymous referee for making several valuable suggestions concerning the introduction and organization of the work. 

\section{The Calculus of Dualizable Patterns}

    \subsection{Definitions and Examples}

        \subsubsection{The Definition of a Dualizable Pattern}
        
            Let $V := \DS21$ with preferred, oriented longitude $\lambda_V := \set{pt} \times \sphere1$ and oriented meridian $\mu_V := \boundary\disk2 \times \set{pt}$ marked in $\boundary V$. A pattern is a smooth embedding $P : \sphere1 \hookrightarrow \Int V$ of an oriented circle into the interior of $V$. We define the winding number, $w(P) := lk(\mu_V, P)$, where we identify $P$ with its oriented image. We further define the wrapping number $\wrap(P) = | P \cap (\disk2 \times \set{pt})|$, minimized over the oriented isotopy class of $P$ transverse to the disk. It is immediately clear that that for any pattern $P$ we have that for some $c \in \N$: $$ \wrap(P) = |w(P)| + 2c $$ where we can interpret $c$ as the number of pairs of algebraically cancelling intersection points between $P$ and $\set{pt}\times\disk2$.
            
            Given a pattern $P$ we denote $V \setminus \nbhd P$ by $V_P$ which comes with two boundary components: $\boundary V$ and $\boundary \nbhd P$ both of which are tori. We call $\boundary V$ the `outer torus' and $\boundary \nbhd P$ the `inner torus' associated to $V_P$. The outer torus comes with a preferred identification to $\sphere1\times\sphere1$ given by the pair $(\mu_V,\lambda_V)$. We can also construct a similar identification for the inner torus as follows: glue $V$ into the complement of the unknot in $\sphere3$ so that $\mu_V$ is glued to the meridian of the unknot, and $\lambda_V$ is glued to the unique longitude which bounds a disk in the complement. This converts $P$ into an embedding $P' : \sphere1 \into \sphere3$, thereby defining an oriented knot. We get a triple embedding, 
            $$
                \nu(P) \into V \into \sphere3 
            $$
            and there is a unique null-homologous longitude on $\boundary \nu(P)$ which bounds an oriented surface in $\sphere3 \setminus \nu(P)$. We call this curve $\lambda_P \subset \boundary\nu(P)$ and we orient it coherently with $P$ so that $lk(\mu_V, \lambda_P) = lk(\mu_V, P)$. After choosing any identification of $\DS21$ with $\nu(P)$, we define $\mu_P$ to be the image of $\boundary\disk2\times\set{pt}$ under this identification. It is a standard exercise in 3-manifold topology that this is well defined. We orient $\mu_P$ so that $lk(\mu_p,P) = +1$. Thus the pair $(\mu_P,\lambda_P)$ define an identification of the inner torus of $V_P$ with the product of oriented circles, $\SS11$.
            
            \begin{definition}\cite[Def.\ 3.1]{miller_knot_2018}
                A pattern $P$ is \emph{dualizable} with dual pattern $P\dual \subset V\dual$ iff there exists an orientation preserving homeomorphism $* : V_P \to V\dual_{P\dual}$ such that,
                    \begin{enumerate}[label=(\roman*),itemsep=3mm]
                        \item $*$ maps the inner torus of $V_P$ to the outer torus of $V\dual_{P\dual}$, and the outer torus to the inner torus.
                        \item $(*(\mu_V),*(\lambda_V)) = (-\mu_{P\dual}, \lambda_{P\dual})$ 
                        \item $(*(\mu_P),*(\lambda_P)) = (-\mu_{V\dual}, \lambda_{V\dual})$
                    \end{enumerate}
            \end{definition}
        
            Given a dualizable pattern $P$, we can obtain the dual pattern by Dehn filling with slope $\infty$ the outer boundary (only) of $V_P$. The existence of $*$ implies that this manifold is homeomorphic to $\DS21$ so we call it $V\dual$ and we obtain $P\dual$ as the core of the new Dehn filling, oriented coherently with $\lambda_V$. 
            
            Whenever we draw explicit examples of patterns, their duals, and satellite constructions, we will adopt the orientation convention that knot complements in $\sphere3$ are oriented with the inward pointing normal first at the boundary, while solid tori are oriented with the outward pointing normal first at the boundary. Thus all our gluing maps are orientation preserving at the boundary. This fixes an orientation convention on the exterior of a pattern $V_P$. If we pick a random companion knot $K$ and consider the decomposition of $\sphere3$ naturally associated to the satellite knot $P(K)$ we see $\sphere3$ split into three oriented, codimension-0 parts:
            $$ 
                \sphere3 = \nbhd P \underset{\boundary \nbhd P}{\cup} V_P \underset{\boundary V}{\cup} E_K
            $$
            Since orientations must agree at both boundary tori across all three pieces and we know that $\nbhd P$ is oriented outward-normal-first, while $E_K$ is oriented inward-normal-first, we can deduce that $V_P$ must be oriented inward-normal-first at $\boundary \nbhd P$ and outward-normal-first at $\boundary V$. This leads to the rather pleasant sounding convention that $V_P$ is oriented outward-first at its outer boundary and inward-first at its inner boundary. In case this is confusing, we include a totally explicit picture of $Wh^-(3_1) \subset \sphere3$ with all longitudes, meridians, and orientation arrows drawn:
            
            \begin{figure}[hpbt]
                \labellist
                \small\hair 2pt
                 \pinlabel \red{$\mu_{Wh^-}$} [l] at 65 77
                 \pinlabel \blue{$\lambda_V$} [l] at 87 147
                 \pinlabel \red{$\mu_V$} [] at 284 118
                 \pinlabel {$V$} [ ] at 19 166
                 \pinlabel \green{$\nu_{Wh^-}$} [l] at 87 117
                 \pinlabel \textcolor{violet}{$\lambda_{Wh^-}$} [l] at 82 101
                 \pinlabel {$E_{3_1}$} [br] at 313 1
                \endlabellist
                \centering
                \includegraphics{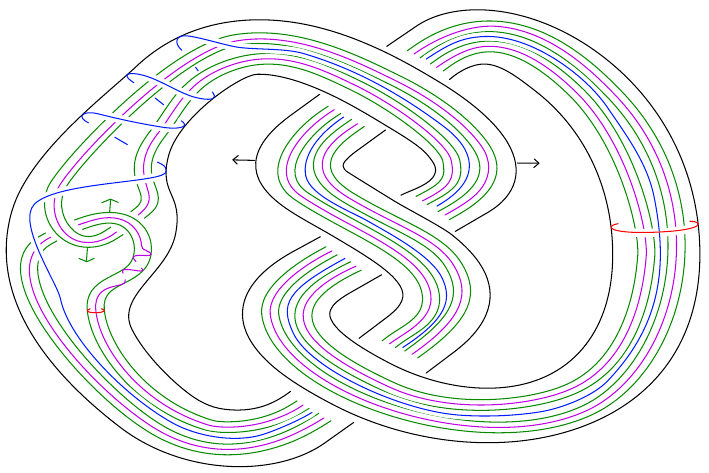}
                \caption{Orientations and surfaces in $Wh^-(3_1)$}
                \label{fig:whpt}
            \end{figure}
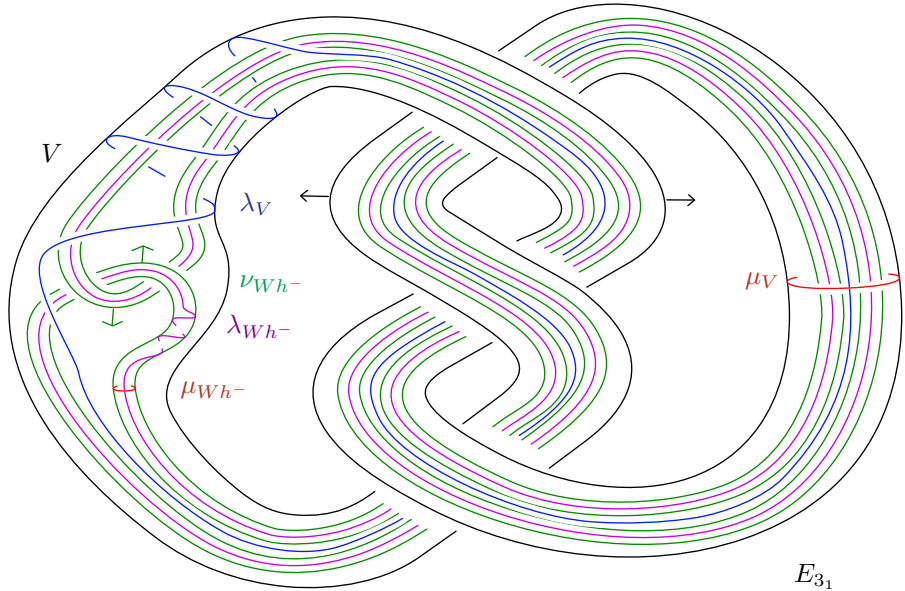
            
            We now summarize some useful facts about dualizable patterns, all of which are stated and proven in \cite{miller_knot_2018}. Given a knot $K \subset \DS21$, we can construct an associated knot, $\widehat{K} \subset \SS21$, by doubling $\DS21$ and letting $\widehat{K}$ be the composition of the inclusions: $$ \widehat{K}: K \into \DS21 \into \SS21 $$ Often $K$ will be either a pattern, $P$, or one of the framing curves from $\boundary V_P$.
            
            \begin{proposition}
            \label{prop:ss}
                \cite[Prop\ 3.5]{miller_knot_2018} A pattern $P$ is dualizable iff $\widehat P$ is oriented-isotopic to $\widehat{\lambda_V} = \set{pt}\times\sphere1$ in $\SS21$. 
            \end{proposition}
            
            \begin{corollary}
            \label{cor:pione}
                \cite[Cor.\ 3.6]{miller_knot_2018} A pattern $P$ with $w(P) = 1$ is dualizable iff $\mu_P \in \llangle \mu_V \rrangle$, the subgroup of $\pi_1(V_P)$ normally generated by $\mu_V$. 
            \end{corollary}
            
            These statements are important because their proofs indicate both how to construct dualizable patterns and, given a dualizable pattern, how to compute its dual pattern. We will give a brief verbal description of these processes here and include an example, \autoref{fig:PD1}. We highly encourage the reader to separate out the pages containing these figures so as to examine them side-by-side with the verbal description below. 
            
            The key to understanding both points is to study the copies of $\SS21$ obtained by doubling the solid tori $V$ and $V\dual$ containing dualizable patterns $P$ and $P\dual$. Each decomposes into three pieces: 
            \begin{gather*}
                \double(V) = \nbhd P\ \cup\ V_P\ \cup\ \DS21 \\
                \double(V\dual) = \nbhd {P\dual}\ \cup\ V\dual_{P\dual}\ \cup\ \DS21
            \end{gather*} 
            where we recall that the pair of curves $(\mu_P, \lambda_P)$ give a natural identification $\nbhd P = \DS21$, and similarly for $\nbhd{P\dual}$. We can equivalently describe this as: 
            \begin{gather*}
                \double(V) = (\DS21)_{Inner}\ \cup\ V_P\ \cup\ (\DS21)_{Outer} \\
                \double(V\dual) = (\DS21)_{Inner}\ \cup\ V\dual_{P\dual}\ \cup\ (\DS21)_{Outer}
            \end{gather*} 
            The homeomorphism $*: V_P \to V\dual_{P\dual}$, which swaps the inner and outer boundaries, carries meridians to meridians and longitudes to longitudes. We can extend it to a map, 
            $$
                \widehat{*} : \double(V) \to \double(V\dual)
            $$
            which swaps the two solid tori filling in the inner and outer boundaries. $\widehat{*}$ is clearly an orientation preserving self-homeomorphism of $\SS21$, since both $V$ and $V\dual$ come with preferred identifications to $\DD21$, and thus their doubles have such identifications to $\SS21$. 
            
            We are going to keep track of what happens to $\widehat{\lambda_V}$ and $\widehat{\lambda_P}$, considered as oriented curves in $\SS21$, under the homeomorphism $\widehat{*}$. When we build $\SS21$ by doubling $V$, we get one of two canonical handle-body diagrams for $\SS21$. The first one, which we call the `2-ball diagram', comes from viewing $\SS21$ as surgery on a pair of disjointly embedded 3-balls. We draw the complement of the two $\ball3$'s and we identify the two boundary $\sphere2$'s of the complement by the map which reflects the complement across the sphere bisecting the line connecting the two boundary spheres perpendicularly. The $\sphere1$ factor in this picture is given by the line connecting the two balls in the diagram, and the $\sphere2$ factor is given by perpendicular bisector. For a more detailed description of this diagram consult \cite[page\ 114-115]{gompf_4-manifolds_1999}. The second diagram, which we will call the `dotted-circle diagram', comes from viewing $\SS21$ as the boundary of zero-surgery on the unknot. We draw a round, circular unknot in $\sphere3$ which we label with a $\langle 0 \rangle$ to indicate performing 0-surgery on it. Here, the $\sphere1$ factor is given by a small meridian to the surgered circle, while the $\sphere2$ factor is given by the disk which the surgered circle bounds. Consult \cite[page\ 167-168]{gompf_4-manifolds_1999} for more details. 
            
            It follows from work of Gluck \cite{gluck_embedding_1962} that there are only two orientation preserving self-homeomorphisms of $\SS21$ up to isotopy, which are the identity and the Gluck twist. Since the Gluck twist fixes the isotopy class of any fixed copy of the $\sphere1$-factor, it follows from the construction of $\hat{*}$ that $\hat{P}$ must be oriented isotopic to $\set{pt} \times \sphere1$ in $\SS21$ by an ambient isotopy. However this isotopy may change the framing of $\hat{P}$ since it may involve sliding $\hat{P}$ over the $\sphere2$ factor. The algebraic intersection number of $\hat P$ with the $\sphere2$ factor is $+1$ so each such slide will change the framing by $\pm 2$. After the ambient isotopy, apply Gluck twisting around $\hat P$ to correct the framing, and let $\hat{P\dual}$ be the image of $\hat{\lambda_V}$ under this isotopy and twisting. We recover $V\dual$ as $\SS21 \setminus \hat{\mu_V}$. This tells us how to dualize a pattern in practice. Given a dualizable pattern $P$, we can obtain $P\dual$ as follows: 

            \begin{enumerate}[itemsep=1mm]
                \item Draw $\hat{P}, \hat{\mu_V}, \hat{\lambda_V}$ in either of the standard diagrams of $\SS21$ (2-ball or dotted-circle) and keep track of their framings.
                \item Observe that removing a neighborhood of $\hat{\mu_V}$ recovers $P \subset V$.
                \item Draw an explicit, ambient isotopy of $\SS21$ which moves $\hat{P}$ to an $\sphere1$ factor, fixes $\hat{\mu_V}$, and keeps $\hat{\lambda_V}$ disjoint from the final position of $\hat{P}$ in the diagram (what this means will become clear in the examples). 
                \item The ambient isotopy may change the framings of $\hat{\lambda_V}$ and $\hat{P}$, so apply the Gluck twist map to the diagram to correct them (again, what this means will become clear in the examples). 
                \item Let $P\dual$ be the image of $\hat{\lambda_V}$ after the ambient isotopy and Gluck twisting, considered in the complement of $\mu_V$.
            \end{enumerate}
            
            The correctness of this process follows from the rough discussion proceeding it, with a little care. A completely detailed account is given in \cite{miller_knot_2018}. We give an example of this computation which we will carry out in dotted-circle notation. We refer the reader to \cite[Figure 2]{miller_knot_2018} for an example in 2-ball notation, or to version one (available on the arXiv) of this paper. 
    
        \subsubsection{An Example of Dualizing a Pattern}
        
            The following sequence of diagrams describe how to construct the dual of a pattern using the method described in the previous subsection. The pattern in question, $P$, is the main ingredient we use, along with its dual, $P\dual$, to construct the knots $K_{n,c}$ in \autoref{thm:main}. We do not actually include a diagram of any of these knots in this paper but we give complete instructions for how to draw them using the patterns $P$ and $P\dual$. This sequence of diagrams shows the isotopy which takes $\hat{P}$ to $\hat{\lambda_V}$ and then the Gluck twisting which we apply to correct the framings.
            
            Any isotopy of a curve in $\SS21$ can be drawn as a sequence of planar isotopies, Reidemeister moves, and slides of the curve over the surgered curve of the diagram. We will not indicate the paths traced out by curves during planar isotopies or Reidemeister moves. We give the start and end diagrams instead. We draw handle slides by using a purple, dashed arrow to indicate the band, which we assume to lie in the plane except where it crosses over/under other strands in the diagram. We will indicate the new framings of the curves by integers in parentheses after each handle slide. We also use a purple arrow to indicate the rotation of an $\sphere2$ factor when we apply the Gluck twist diffeomorphism to correct framings at the end. The component corresponding to the surgery diagram of $\SS21$ is drawn in red, $\hat{P}$ is drawn in black, and $\hat{\lambda_V}$ is drawn in green. $P\dual$ may be obtained from the final diagram by removing a neighborhood of the red curve from $\sphere3$. 

            \newpage
            
            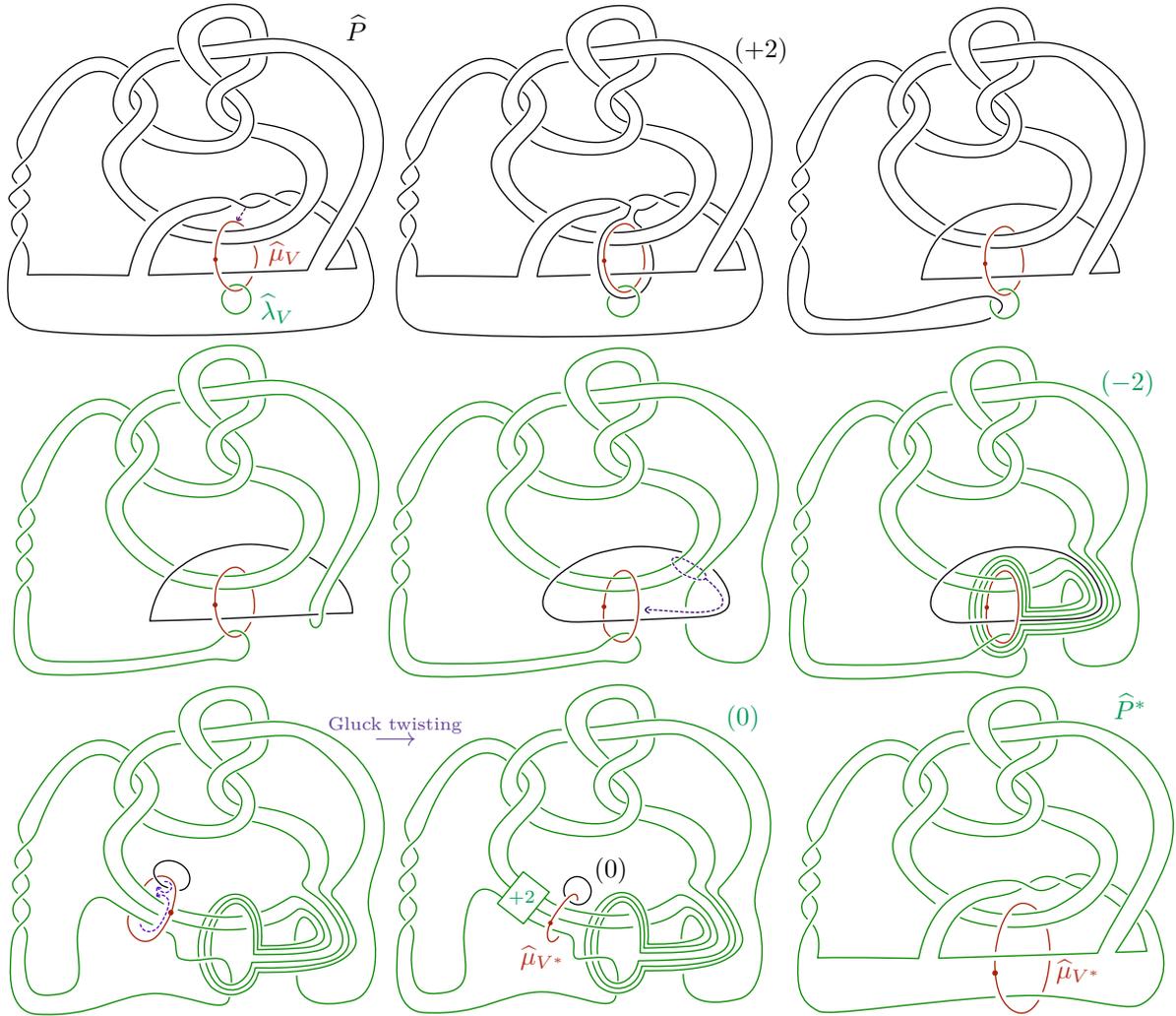
\begin{figure}[htpb]
                \labellist
                \small\hair 2pt
                 \pinlabel \green{$\hat \lambda_V$} [l] at 96 278
                 \pinlabel {$\hat P$} [bl] at 129 379
                 \pinlabel \red{$\hat \mu_V$} [l] at 99 299
                 \pinlabel {$(+2)$} [bl] at 276 370
                 \pinlabel \green{$(-2)$} [bl] at 415 243
                 \pinlabel \green{\tiny{$+2$}} [ ] at 197 53
                 \pinlabel {$(0)$} [bl] at 223 57
                 \pinlabel \green{$(0)$} [bl] at 273 115
                 \pinlabel \green{$\hat P\dual$} [bl] at 420 120
                 \pinlabel \purple{$\buildrel{\text{Gluck twisting}}\over\longrightarrow$} [t] at 149 125
                 \pinlabel \red{$\hat \mu_{V\dual}$} [ ] at 205 30
                 \pinlabel \red{$\hat \mu_{V\dual}$} [ ] at 408 25
                \endlabellist
                \centering
                \includegraphics{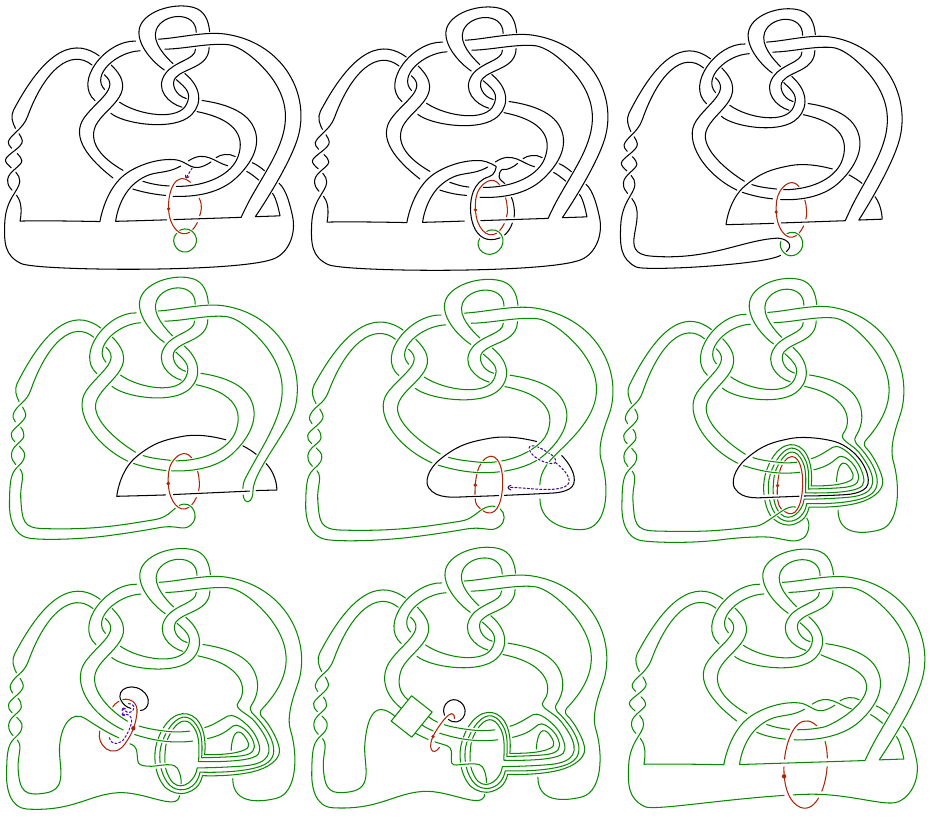}
                \caption{Computing the dual of $P$. The transition between the final two pictures is an isotopy of the green component in the complement of the other two. $P\dual$ is given by the curve $\hat P\dual$ in the complement of a neighborhood of the red curve in the final picture.}
                \label{fig:PD1}
            \end{figure}

        \subsubsection{A Universal Construction}
            \begin{proposition} 
            \label{prop:unidual}
                Any pattern, $P$, constructed as follows is dualizable: 
                    \begin{enumerate}[itemsep=1mm]
                        \item Pick an oriented knot $K_0$ in $\sphere3$ and draw a positively oriented meridian, $m_0$, around $K_0$.
                        \item Draw unoriented parallel copies of $m_0$, labelled $m_1, \ldots, m_n$, so that all $m_{0 \le i \le n}$ bound disjoint spanning disks. 
                        \item Draw a band $b_i$ connecting $K_0$ to $m_i$ for each $1 \le i \le n$. The bands are allowed to link in any way with $K_0$ and $m_0$, but they are not allowed to intersect each other. 
                        \item Orient the new knot $K'$ using the induced orientation from $K_0$. 
                        \item Let $P \subset V$ be the knot $K' \subset \sphere3 \setminus \nbhd{K_0}$. 
                    \end{enumerate}
                Moreover, any dualizable pattern may be obtained from this construction by taking $K_0$ to be the unknot and choosing the bands correctly.
            \end{proposition}
            
            We remark that although taking the base knot to be the unknot is sufficient to produce any dualizable pattern, it is often more convenient to start with some non-trivial knot.
            
            \begin{proof}
                Let $P$ be obtained from the process of \autoref{prop:unidual}. We start by showing that $P$ is dualizable. It suffices to give an explicit isotopy from $\hat{P}$ to $\sphere1$ in $\SS21$ by \autoref{prop:ss}. We consider this problem in `dotted-circle' diagrams. We obtain a picture of $\hat{P} \subset \SS21$ by drawing $K' \cup m_0$ and doing 0-surgery on $m_0$. Since the ends of the bands $b_*$ are the parallel copies of $m_0$, we can slide each end once over the 0-surgered component $m_0$. Each band then retracts back onto the original knot $\hat{K_0}$ so we see that $\hat{K'}$ is isotopic to $\hat{K_0}$ in $\SS21$. It follows from the classical light bulb theorem in $\SS21$ that ${K'}$ is isotopic to $\sphere1$ since $\abs{\hat{K_0} \cap \sphere2} = 1$.
                
                Let $P$ be any fixed dualizable pattern. We must find a collection of bands to attach to $K_0 = U$ as in \autoref{prop:unidual} which recover $P$. Consider $\hat{P}$ in the `dotted circle' diagram of $\SS21$. We know there is an ambient isotopy $\mathcal I : \SS21 \times [0,1] \to \SS21$ such that $\mathcal I_0 = \Id_{\SS21}$ and $\mathcal I_1(\hat{P}) = \sphere1$ by the characterization of dualizable patterns given in \autoref{prop:ss}. Using standard techniques from Kirby calculus applied in dimension 3 to $\mathcal I$, we see it can be represented in the `dotted circle' diagram as an alternating sequence of slides over the 0-surgered region and isotopies in the complement of the 0-surgered region. The key observation is that we can modify the isotopy in $\SS21$ so that all the handle slides occur in sequence at the beginning, followed by an isotopy in the complement of the 0-surgered region. 
                
                We can assume the first move in the sequence associated to $\mathcal I$ is a handle slide without loss of generality. We proceed down the sequence until we come to the first isotopy, which we call $I_1$. At the end of $I_1$, there is a small arc on $\hat{P}$ where the next band for the next handle slide, $S_1$, is attached. We can pull this arc back along $I_1$ and keep track of the path it traces out in $\SS21$. It may be helpful to visualize this path as a ribbon. Since this path itself is contained in a neighborhood of an arc, we can modify $I_1$ so that the path intersects neither itself nor the knot $\hat{P}$. We create a new sequence, defining a new isotopy, by removing $I_1$ and proceeding directly to $S_1$ except we append to the band the path traced out by its attaching arc under $I_1$, thereby defining a different slide $S'_1$. The result of $S'_1$ will be isotopic to the result of $S_1 \circ I_1$ by an isotopy called $I'_1$. We now construct a new isotopy $\mathcal I'$ by removing the subsequence $I_1$, $S_1$ from $\mathcal I$ and replacing it with $S'_1$, $I'_1$. By proceeding inductively in this manner, we can create a new isotopy which consists of a sequence of slides of $\hat{P}$ over the 0-surgered region, followed by a single isotopy in the complement. 
                
                \begin{figure}[t]
                    \labellist
                    \small\hair 2pt
                     \pinlabel {$b_n$} [l] at 114 153
                     \pinlabel {$b_{j<n}$} [l] at 166 135
                     \pinlabel {$K_0$} [l] at 59 28
                     \pinlabel {$b_n^\prime$} [l] at 365 152
                    \endlabellist
                    \centering
                    \includegraphics{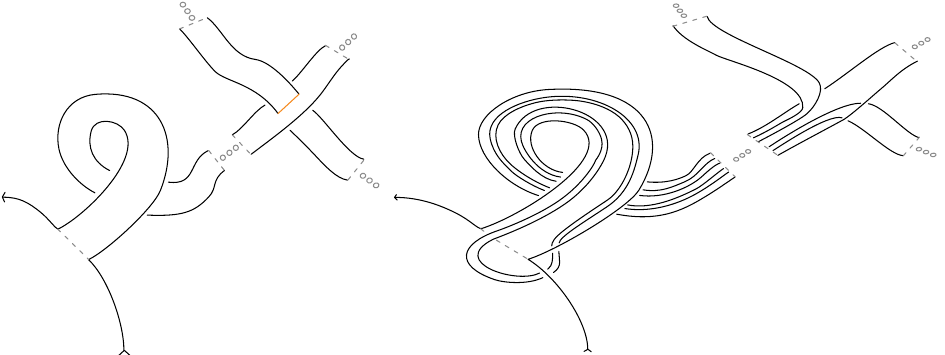}
                    \caption{Resolving the intersections between the $n^{th}$ band and those of lower index.}
                    \label{fig:bandfinger}
                \end{figure}
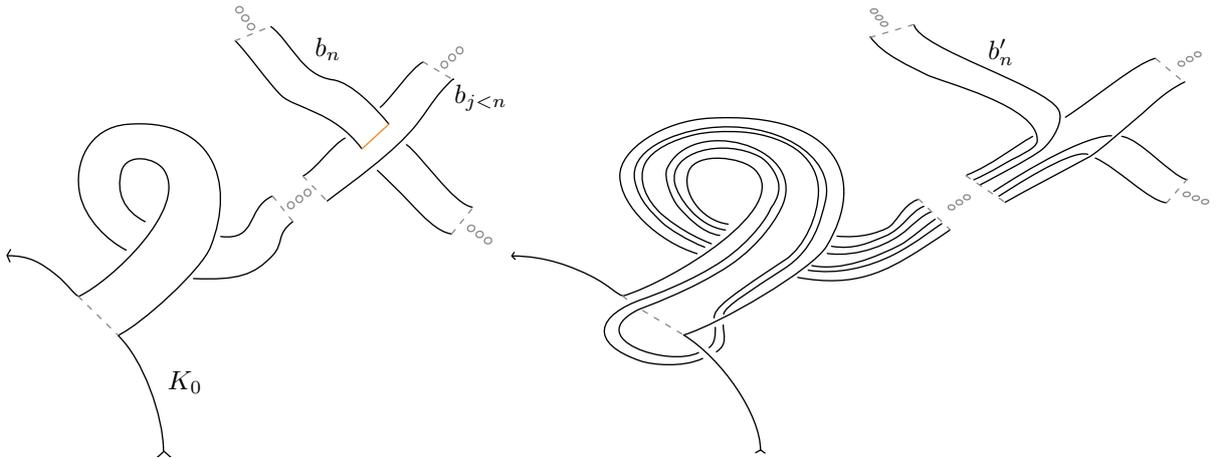    
                
                We now have $\mathcal I'$, which consists of a sequence of handle slides followed by an isotopy in the complement of the 0-surgered region. Because the slides are done one after another and not simultaneously, it is possible that the band from the $n^{th}$ slide might intersect the bands from the slides $1,\ldots, n-1$. We can fix this by resolving the intersection of each band with all those of lower index as in \autoref{fig:bandfinger}. This makes the bands significantly more complicated but it ensures that they are disjoint. We can now play the isotopy backwards to see the $\sphere1$ factor of $\SS21$ being simultaneously slid some number of times over the 0-surgered region to produce the curve $\hat P$. This shows the existence of a sequence of bands ${b_1, \ldots, b_n}$ such that if we take $K_0$ to be the unknot (really we should think of it as the $\sphere1$ factor) and attach the bands as in \autoref{prop:unidual} then we obtain $P$.
            \end{proof}
            
            Since this construction happens in $\SS21$, it is naturally compatible with the method we give for computing the dual pattern. Given $K_0$ and $b_1, \ldots, b_n$, lift them to $\SS21$, viewed as zero surgery on $m_0$ as before. Notice that $\lambda_V$ appears as a small meridian of $m_0$, but it does not link with the parallel copies $m_1, \ldots, m_n$ which make up the ends of the bands $b_1, \ldots, b_n$. We slide each band over $m_0$ which causes the end of each band to link once, geometrically, with $\lambda_V$. We retract the bands back onto $K_0$, which effectively attaches the same set of bands to $\lambda_V$. We slide $K_0$ over $m_0$ until it has become an unknot (possibly further tangling the bands now attached to $\lambda_V$). We slide every strand of $\lambda_V$ which passes through $K_0$ over $m_0$ then apply the Gluck twist (spinning the sphere composed of the obvious spanning disk for $m_0$ with the core of the 0-surgery) to correct the framings. A very simple example of this is given in \autoref{fig:basewithbandstodual}.
            
            \begin{figure}[!ht]
                \labellist
                \small\hair 2pt
                 \pinlabel \green{$\hat \lambda_V$} [ ] at 14 248
                 \pinlabel \red{$m_0$} [ ] at 40 240
                 \pinlabel {$m_1$} [l] at 90 211
                 \pinlabel \blue{$K_0$} [ ] at 94 245
                 \pinlabel \orange{$b_1$} [l] at 136 273
                 \pinlabel \blue{$(+2)$} [l] at 296 267
                 \pinlabel \green{$(-2)$} [ ] at 17 107
                 \pinlabel \purple{Gluck twisting} [t] at 310 16
                 \pinlabel \green{\tiny{$+2$}} [ ] at 348 98
                 \pinlabel \green{$(0)$} [ ] at 332 134
                 \pinlabel \blue{$(0)$} [ ] at 396 111
                 \pinlabel \green{$\hat P\dual$} [l] at 433 124
                 \pinlabel \red{$\hat \mu_{V\dual}$} [l] at 393 87
                 \pinlabel {$\hat P$} [ ] at 47 158
                \endlabellist
                \centering
                \includegraphics{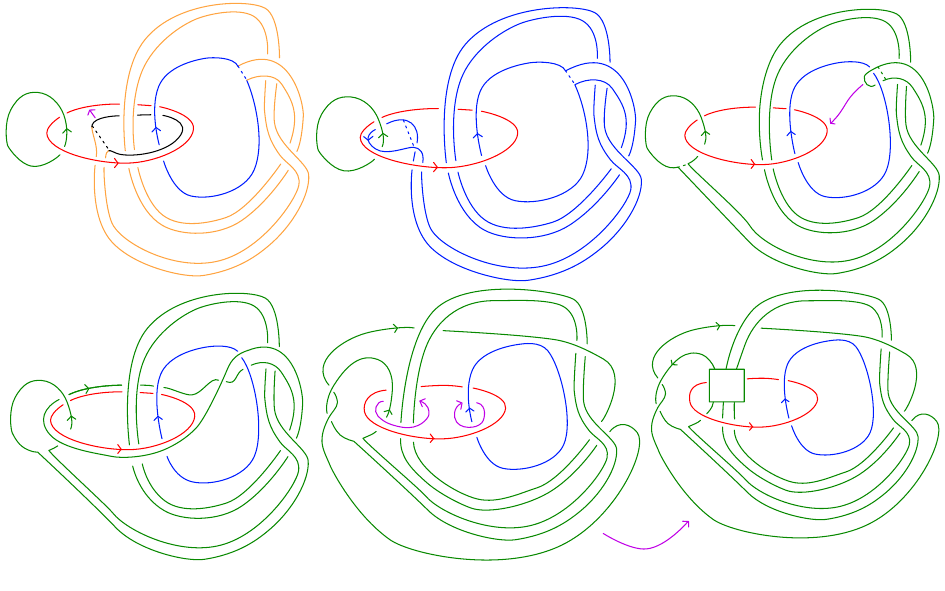}
                \caption{Working out $P\dual$ from a base-knot-with-bands description of $P$.}
                \label{fig:basewithbandstodual}
            \end{figure}
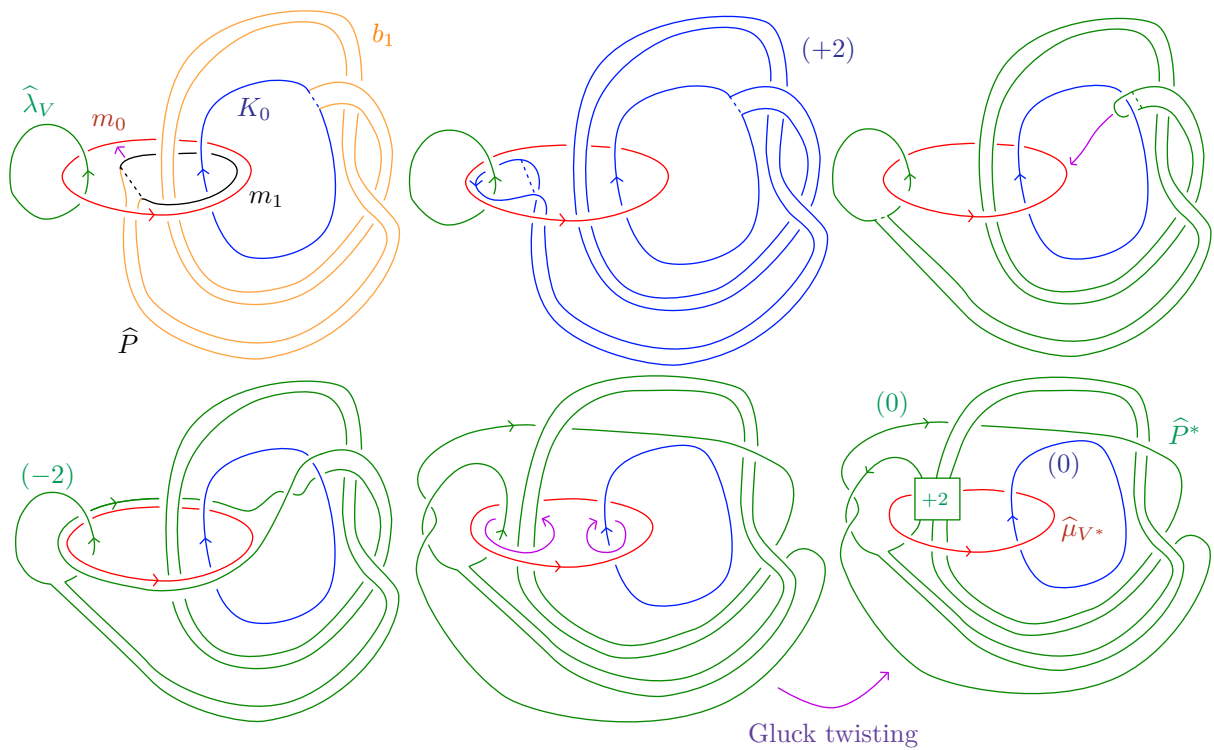

    \subsection{Properties of Dualizable Patterns}

    We will define all of the notation we need to manipulate dualizable patterns and knot invariants which are compatible with them in this section. These objects and their properties form a pleasing calculus which we wish to express in its entirety. We include a singe-page cheat sheet at the end of the paper which can be printed out or clipped separately from the rest. We hope it may serve as a useful reference. 
    
    \subsubsection{Definitions and Notation}
    
    \begin{definition}
        Let $K \subset \sphere3$ be a knot with exterior $E_K$ and let $P \subset V$ be a dualizable pattern. Let $E_K$ come with oriented meridian $\mu_K$ and oriented, null-homologous longitude $\lambda_K$ on its boundary. Observe that $E_K \cup V$ is diffeomorphic to $\sphere3$ if we glue the pair $(\mu_K, \lambda_K)$ to $(\mu_V, \lambda_V)$. Let the knot $P(K)$ be the image of $P$ under this diffeomorphism. We call the knot $P(K)$ the \emph{satellite knot} with \emph{pattern} $P$ and \emph{companion} $K$. 
    \end{definition}
    
    If we have a sequence of patterns $P_1, \ldots, P_k$, we can apply them iteratively to a companion knot $K$ to obtain the knot, $P_1 \circ \ldots \circ P_k (K)$. If we draw this knot in $\sphere3$ in the conventional way, we see that the leftmost pattern in the composition notation, $P_1$, appears to be the innermost `shape' in the diagram of the satellite knot. Conversely, the companion, $K$, appears to be the outermost `shape'. 
    
    \begin{definition}
        Given a pattern $P$, let $P_n$ denote the pattern obtained by taking the image of $P$ under the self-diffeomorphism of $V$ which twists a single $\disk2$ factor $n$ times in the direction specified by the orientation of $\mu_V$. 
    \end{definition}
    
    \begin{definition}
        Given a pattern $P$, let $\bar{P}$ denote the pattern obtained by reversing the orientation of $V$, $P$, $\lambda_V$, and $\mu_V$. We formally define the new orientation to be outward-normal-first just as $P$ was. This is equivalent to reversing all the crossings and orientation arrows in a diagram for $P$. 
    \end{definition}
    
    \begin{definition}
        Given a pattern $P$, let the same symbol, $P$, also denote the knot $P(U)$ in appropriate contexts. For example, if $Q$ is another pattern then $Q(P)$ denotes the knot $Q \circ P(U)$. 
    \end{definition}
    
    \begin{definition}
        Given a knot $K \subset \sphere3$, let $K\pound$, called `K pound', denote the unique wrapping number one pattern such that $K(U) = K$. 
    \end{definition}    
    
    We often combine this definition with the previous one, so that for a general pattern, $P$, $P\pound$ is the unique wrapping number one pattern such that $P\pound(U) = P = P(U)$. The action of a pound pattern is to take the connected sum of the companion knot with $P(U)$, which explains the notation. 
    
    \begin{definition}
        Given a dualizable pattern $P$, let $P\dual$ denote the dual pattern with outward normal first orientation (as usual). 
    \end{definition}
    
    \begin{definition}
        Given a pattern $P$, let $P^n$ denotes the pattern $P$ composed with itself $n$ times. 
    \end{definition}
    
    \begin{definition}
        Given a dualizable pattern $P$, let $P\invs$ denote the pattern $\bar{P\dual}$. The name comes from the fact that $P\invs(P)$ and $P(P\invs)$ are always concordant to the unknot. 
    \end{definition}
    
    Since we often wish to combine one or more of these operations, we adopt the following conventions concerning their order in the absence of parentheses: 
    
    \begin{enumerate}[itemsep=1mm]
        \item $P\dual_n = (P\dual)_n$ not $(P_n)\dual$ (star goes inside twisting)
        \item $\bar{P}_n = (\bar{P})_n$ not $\bar{(P_n)}$ (bar goes inside twisting)
        \item $P_n^m = (P^m)_n$ not $(P_n)^m$ (composition goes inside twisting)
        \item $P\dual\pound = (P\dual)\pound$ not $(P\pound)\dual$ (star goes inside pound)
        \item $\bar{P}\pound = (\bar{P})\pound$ not $\bar{(P\pound)}$ (bar goes inside pound)
        \item $P\pound^m = (P^m)\pound$ not $(P\pound)^m$ (composition goes inside pound)
        \item These conventions are chosen so that if a pattern is written with a sequence of superscripts and a sequence of subscripts, then the operations should be applied in the order: superscripts from left to right followed by subscripts from left to right (bar counts as the leftmost superscript!). 
    \end{enumerate}
    
    Now that we have clearly established the notation we give several relations, all of which are proven, if not in this notation, in \cite{miller_knot_2018}.
    
    \begin{proposition}
        \label{prop:properties}
        The following properties hold for all dualizable patterns $P,Q$, knots $K,J$, and $n, m \in \Z$:
        \begin{enumerate}[label=(\roman*),itemsep=1mm]
            \item $K\pound(J) = K \connsum J = J \connsum K = J\pound(K)$, \hspace{5mm} $P(K) = P(K\pound(U)) = (P \of K\pound)(U) = P \of K\pound$
            \item $(P\dual)\dual = P$, \hspace{5mm} $(P_n)_m = P_{(n+m)}$, \hspace{5mm} $P = P_0$, \hspace{5mm} $\bar{(\bar P)} = P$
            \item $P\dual_n = (P\dual)_n = (P_{-n})\dual$, \hspace{5mm} $\bar{(P_n)} = (\bar{P})_{-n}$
            \item $P\invs = \bar{(P\dual)} = \bar{P}\dual$
            \item $(P \of P\invs)(K) \sim K \sim (P\invs \of P)(K)$
            \item $(P \of Q)\dual = Q\dual \of P\dual$
            \item $(P \of Q)_n = (P_n \of Q_n)$
            \item If the wrapping number of $P$ is one, then: $$P\dual = P, \hspace{5mm} P\invs = \bar{P}, \hspace{5mm} P_n = P, \hspace{5mm} P = P(U)\pound, \hspace{5mm} P(K) = P(U) \connsum K$$
        \end{enumerate}
    \end{proposition}
    
    This proposition shows that conventions (3) and (5) are actually unnecessary since the operations they govern commute. Nevertheless, we include these conventions for completeness. 
    
    \subsubsection{Traces and Surgeries}
    
    \begin{definition}
    \label{def:trace}
    Let $K \subset \sphere3$ be an oriented knot, then let $X_n(K) := h^0_4 \cup h^2_4(K,n)$, the union of a 4-dimensional 0-handle and a 4-dimensional 2-handle attached along the knot $K$ with $n$-framing. We denote the boundary of $X_n(K)$ in the usual way by $\sphere3_n(K)$. 
    \end{definition}

    \begin{theorem}
    \label{thm:dualtrace}
    $X_0(P) \diffeo X_0(P\dual)$ for any dualizable pattern $P$. It follows that $\sphere3_0(P) \diffeo \sphere3_0(P\dual)$. 
    \end{theorem}
    
    The proof consists of checking that both three manifolds can be obtained as fillings of $V_P$ and $V\dual_{P\dual}$ respectively, and that $*$ extends to these fillings. In fact this diffeomorphism extends to a diffeomorphism of the 4-dimensional traces. Understanding this is a good exercise for the reader, and the details are written in \cite{miller_knot_2018}. Once the reader understands the argument for $n=0$, it should be reasonably clear how to adjust the result to arbitrary $n$. The stronger theorem is:
    
    \begin{theorem}
    \label{thm:dualntrace}
    $X_n(P) \diffeo X_n(P\dual_n)$ for any dualizable pattern $P$ and $n \in \Z$. Thus $\sphere3_n(P) \diffeo \sphere3_n(P\dual_n)$. 
    \end{theorem}
    
    The the reader may be surprised that the left hand side is not $X_n(P\dual_{-n})$. The trick is that the sign of $n$ changes twice during the natural process by which we construct the diffeomorphism. The first change comes from the fact that $* : V_P \to V\dual_{P\dual}$ reverses the orientations of the meridians but not the longitudes, which sends the framing $n \mapsto -n$. This shows $X_n(P) \diffeo X_n((P_{-n})\dual)$. The second sign change comes from reversing the order of the twisting and the dualization. 
    
    Already, it should be clear that combining \autoref{prop:unidual}, \autoref{prop:properties}, and \autoref{thm:dualntrace} gives a formidable repertoire of tools with which we can construct pairs of knots $(K,J)$ such that $X_n(K) \diffeo X_n(J)$ for some fixed $n \in \Z$. When $(K,J)$ share this property, we will call $J$ an \emph{n-retrace} of $K$, and the act of constructing $J$ from $K$ will be called \emph{n-retracing} $K$. 
    
    \subsubsection{Compatible Invariants}
    
    \begin{definition}
    \label{def:compatibleinvariant}
    Let $\mathscr I$ be an knot invariant taking values in an Abelian group $\mathcal A$ (usually $\Z$ in practice). We say $\mathscr I$ is \emph{compatible} with dualizable patterns iff
        \begin{enumerate}[itemsep=2mm]
            \item $\mathscr I : \mathcal C \to \mathcal A$ is a well-defined homomorphism from the smooth (or TOP) concordance group.
            \item If $X_0(K) \diffeo X_0(J)$ then $\mathscr I(K) = \mathscr I(J)$. 
        \end{enumerate}
    for any two knots $K,J \subset \sphere3$.
    \end{definition}
    
    The point is that these two properties combine to give a satellite formula for $\mathscr I$ whenever the pattern used in the satellite construction is dualizable. 
    
    \begin{theorem}
    \label{thm:satformula}
    Let $\mathscr I$ be a knot invariant which is compatible with dualizable patterns then for any dualizable pattern $P$ and companion knot $K$ we have, $$\mathscr I(P_n(K)) = \mathscr I(K) + \mathscr I(P_n) $$ in particular for $n=0$ we have, $$ \mathscr I(P(K)) = \mathscr I(K) + \mathscr I(P)$$ where the pattern with no input should be interpreted as the pattern applied to the unknot. 
    \end{theorem}
    
    \begin{proof}
    Assume the hypotheses of the theorem. We begin by calculating:
        \begin{align*}
            X_0(P_n(K)) &\diffeo X_0(P_n \of K\pound)\\
                        &\diffeo X_0((P_n \of K\pound)\dual)\\
                        &\diffeo X_0(K\pound\dual \of (P_n)\dual)\\
                        &\diffeo X_0(K\pound \of (P_n)\dual)\\
                        &\diffeo X_0(K \connsum (P_n)\dual)
        \end{align*}
    therefore $K \connsum (P_n)\dual$ is a 0-retrace of $P_n(K)$, and it follows from the two properties included in the definition of compatibility that, 
    $$
        \mathscr I(P_n(K)) = \mathscr I(K \connsum (P_n)\dual) = \mathscr I(K) + \mathscr I((P_n)\dual) = \sI(K) + \sI(P_n) 
    $$
    The last equality comes from the fact that $X_0((P_n)\dual) \diffeo X_0(P_n)$ so $\sI((P_n)\dual) = \sI(P_n)$. Taking $n=0$ recovers the second statement in the theorem. 
    \end{proof}
    
    So when is an invariant compatible with dualizable patterns? The best known example of such an invariant is the classical signature and, more generally, the Levine-Tristram signature at a root of unity, $\sigma_K(\omega)$. Both are compatible with dualizable patterns, which can be easily deduced from two of their many, equivalent definitions. That both are concordance homomorphisms comes from defining them as the signatures of certain branched covers of the knot $K$ with coefficients computed in $\Q[\omega]$. One checks that connect-summing the knots corresponds to connect-summing the covers and then quotes Novikov additivity for signatures. That both are invariant under 0-retracing follows from defining the signature in terms of the $\omega$-twisted cohomology of $\sphere3_0(K)$. The details can be found in the excellent survey article \cite{conway_levine-tristram_2021}.
    
    We end the section by encouraging the reader to make use of the one-page cheat sheet (\autopageref{cheatsheet}) which we have appended to this paper. It summarizes the whole calculus of dualizable patterns including the construction and dualization process. We recommend the reader to separate it from the rest of the paper and use it as a side-by-side reference to follow the calculations in the coming sections. 

\newpage
\section{Shake Genus and Complexity Bounds}

    \subsection{Definitions of the Shake-Genus and Complexity}

        \begin{definition}
        \label{def:shgenus}
            We define the \emph{$n$-shake-genus} of a knot $K$, denoted $g^{sh}_n(K)$, to be the minimum genus over all orientable surfaces embedded in $X_n(K)$ representing the canonical generator of $H_2(K,\Z)$ (we assume $K$ is oriented). A knot is \emph{$n$-shake-slice} iff its $n$-shake-genus is zero. 
        \end{definition}

        This definition is well known already in the literature and leads to other definitions in a similar spirit such as $n$-shake-concordance, see \cite{cochran_shake_2016}. 

        \begin{definition}
        \label{def:complexity}
            Let $K$ be an $n$-shake-slice knot. We define the \emph{$n$-complexity} of $K$, denoted by $c(K,n)$ or just $c(K)$ when $n$ is understood from the context, to be the minimum non-negative integer, $c$, satisfying,
            $$ 
                \abs{S \cap \cocore(h^2_4(K,n))} = 1 + 2c 
            $$
            occurring among all smoothly embedded spheres, $S$, transverse to the cocore and generating the second homology of $X_n(K)$. We can also distinguish between whether the sphere $S$ is smoothly embedded or topologically-locally-flatly embedded, which allows us to define smooth and topological complexity respectively.
        \end{definition}

        We note that this definition is a renormalization of the \emph{n-shaking-number} \cite{feller_embedding_2021}. Since the oriented intersection number $S \cdot \cocore(h^2_4) = 1$ for all $S$ we see that the extra intersections must occur in oppositely oriented pairs and $c$ is normalized to count the number of pairs. This definition was inspired by the similar notion of complexity for h-cobordisms defined in \cite{morgan_complexity_1999}. 

    \subsection{4-Genus Bounds Complexity}

        \begin{proposition}\cite{feller_embedding_2021}[Prop. 8.8]
            \label{prop:genusboundscomplexity}
            Let $K$ be an $n$-shake-slice knot then $g^4(K) \le c(K)$. This holds both in the smooth and topological setting. 
        \end{proposition}

        \begin{proof}
            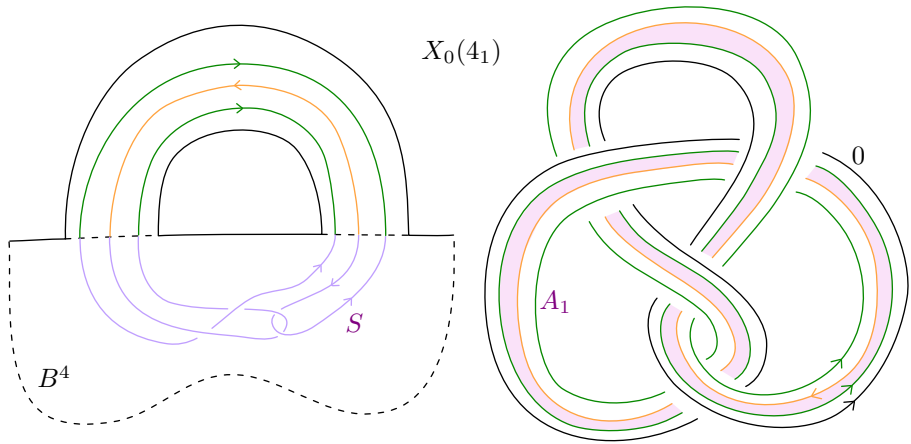
\begin{figure}[b]
                \labellist
                \small\hair 2pt
                 \pinlabel {$X_0(4_1)$} [ ] at 170 149
                 \pinlabel {$\ball4$} [b] at 18 21
                 \pinlabel \textcolor{violet}{$S$} [ ] at 130 47
                 \pinlabel {$0
                 $} [ ] at 318 110
                 \pinlabel \textcolor{violet}{$A_1$} [ ] at 205 55
                \endlabellist
                \centering
                \includegraphics{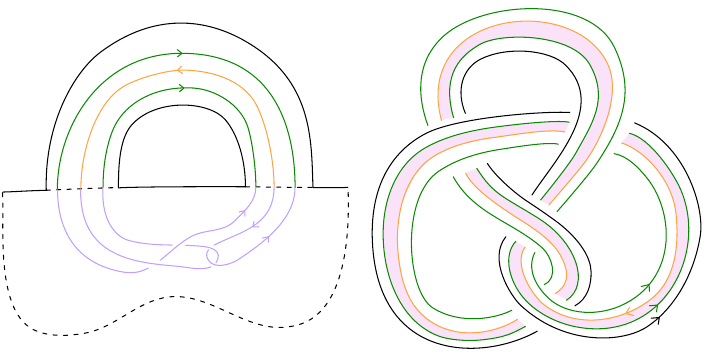}
                \caption{Left: a schematic of the perturbed sphere in the trace. Right: the link $K^0_1$ for $K = 4_1$. We can connect the inner components with the obvious oriented band (purple).}
                \label{fig:shakeslice}
            \end{figure}

            The proof is identical in both the smooth and topological setting so pick one without loss of generality. Let $K$ be as in the proposition and let $S$ be an embedded sphere realizing the minimal complexity. Assume to the contrary that $c(K) < g(K)$. We will use $S$ to construct a surface $\Sigma \subset \ball4$ with boundary $K$ and $g(\Sigma) = c(K)$ thereby creating a contradiction. By working in a chart which contains the 2-handle of $X_n(K)$ in its interior and applying general position to $S$ and the cocore of the 2-handle, we can assume that after a small perturbation $S$ intersects the 2-handle in $2c(K)+1$ parallel copies of the core. The homological condition on $S$ implies that exactly $c(K)+1$ of these parallel copies must be oriented coherently with the core and the remaining $c(K)$ copies must be oriented oppositely to the core so that the oriented intersection number $S \cdot \cocore(h^2_4(K,n)) = 1$. If we consider the boundary of the 4-ball to which the 2-handle is attached, we observe that the parallel copies of the core intersect the boundary of the 4-ball in parallel copies of the knot, twisted so that the linking number of each with the knot is $\pm n$ depending on the orientation of the copy (as unoriented parallels they all twist with the same handedness). Let this link be called $K^c_n \subset \sphere3$. Consider $S \cap \ball4$. We know $S \cap h^2_4(K,n)$ is a disjoint union of $2c(K)+1$ disks so it follows that $S \cap \ball4$ is a $(2c(K)+1)$-times punctured sphere whose boundary is $K^c_n$. We can pair up all but one of the components of $K^n_c$ into oppositely oriented pairs, with each pair bounding a disjoint annulus $A_i$ for $1 \le i \le c(K)$ in the complement of the link. This is clear from \autoref{fig:shakeslice}. Let $\Sigma := (S \cap \ball4) \cup (\bigcup_i A_i)$. It is immediately clear that $\boundary\Sigma = K$, that $\Sigma \subset \ball4$, and that $g^4(\Sigma) = c(K)$. 
        \end{proof}

    \subsection{A Sufficient Condition for High Complexity}

        We now give a general process for constructing $n$-shake-slice knots of arbitrary genus and complexity using dualizable patterns and a compatible invariant. 

        \begin{theorem}
            \label{thm:highcxtyconstruction}
            Let $\sI$ be an integer-valued, compatible invariant whose magnitude provides a lower bound on the 4-genus of a knot. Let $P$ be a dualizable pattern with $P(U)$ smoothly slice, and choose $n \in \Z \setminus \set{0}$, $c \in \N$. It follows that the knot $K$ defined by,
                $$ K := (P\dual_n)^c $$
            is smoothly $n$-shake-slice with 4-genus bounded by,
                $$ c\abs{\sI(P_{-n})} \le g_4(K) \le c \cdot g_3(P\dual \subset \DS21)$$
            and complexity, $c(K,n)$, bounded by,
                $$ c \abs{\sI(P_{-n})} \le c(K,n) \le \frac{1}{2}\left(\wrap(P)^c-1\right)$$
            The lower bounds hold in whichever category $\sI$ bounds the 4-genus, and the upper bounds are smooth. Here, $\wrap(P)$ denotes the unsigned intersection number $\abs{P\ \cap\ \disk2\times\set{pt}}$, and $g_3(P\dual \subset \DS21)$ denotes the minimal genus of a surface in $\DS21$ which spans $P\dual$ and the longitude of $\boundary \DS21$. 
        \end{theorem}

        We will prove \autoref{thm:main} by an application of this theorem with the following data: $P$ will be the wrapping number three pattern exhibited in \autoref{fig:PD1}, and $\sI := \half\sigma(-,t)$ will be the half the Levine-Tristram signature at a particular complex number $t$ (which will depend on $n$). The pattern $P$ was constructed so that $P(U)$ is smoothly slice, $\half\sigma(P_n,t) = \pm1$ for all non-zero $n$, and the 3-genus of $P\dual$ in its solid torus is one, which makes the genus bounds both sharp. The lower bound for complexity will follow from the genus and the upper bound will come from an explicitly constructed smoothly embedded sphere in the trace. 

        \begin{proof}
            Assume the hypotheses of the theorem and let $K$ be given: we begin by showing $K$ is smoothly $n$-shake-slice.  The knot $P^c(U)$ is smoothly slice since $P(U)$ is smoothly slice by hypothesis. Using the calculus of dualizable patterns we show that $K$ is an $n$-retracing of it:
            $$
                X_n(P^c) \diffeo X_n((P^c)_n^*) \diffeo X_n((P^*_n)^c)
            $$
            Clearly $X_n(P^c)$ has a sphere generating its second homology group, which is given by capping off a slice disk for the knot with the core of the 2-handle. Taking the image of this sphere under the diffeomorphism $d:X_n(P^c) \to X_n((P\dual_n)^c)$ shows that $K := (P\dual_n)^c$ is $n$-shake-slice. Next we compute $\sI(K)$ as follows:
            \begin{align*}
                \sI((P\dual_n)^c) = c\sI(P\dual_n) &= c\sI((P_{-n})\dual) = c\sI(P_{-n})\\
                \implies \abs{\sI(K)} &= c\abs{\sI(P_{-n})} \le g_4(K)
            \end{align*}
            and it follows from \autoref{prop:genusboundscomplexity} that $c\abs{\sI(P_{-n})} \le c(K,n)$.
            
            We now turn to the upper bounds, starting with the genus. $P\dual$ is a winding number one pattern, so $g_3(P\dual \subset V\dual) =: g$ is well defined. Let $F\sub{P\dual} \subset V\dual$ be a genus $g$ surface spanning the two-component link $P\dual \cup \lambda_{V\dual}$. For any companion knot $J$, with Seifert surface $F_J$ we can build a new Seifert surface for $P\dual_0(J)$ from $F_J$ and $F\sub{P\dual}$. The new surface $F\sub{P\dual(J)}$ will have genus,
            $$
                g(F\sub{P\dual(J)}) = g(F_J) + g
            $$
            and is constructed as follows: $F_J \cap \boundary\nbhd J = \lambda\sub{\nbhd J}$ which maps to $\lambda\sub{V\dual}$ when we construct $P\dual(J)$ by identifying $V\dual$ with $\nbhd J$. It follows that we can glue together the surfaces $F_J$ and $F\sub{P\dual}$ along $\lambda\sub{V\dual}$ to obtain the surface $F\sub{P\dual(J)}$ with the desired genus. We apply this construction inductively to obtain a genus $c \cdot g$ Seifert surface for ${P\dual}^c$. Next we observe that because $P\dual$ is winding number one, the knots $({P\dual}^c)_n$ and $({P\dual}_n)^c$ are the same. Therefore if we apply $n$ full twists to the entirety of ${P\dual}^c$ we obtain an immersed spanning surface which only has ribbon and circle self-intersections. These self-intersections can be resolved by pushing the surface into $\ball4$, so we have exhibited a genus $c\cdot g$ surface embedded in $\ball4$ which spans $(P\dual_n)^c$. This establishes the upper bound from the theorem.
            
            We can bound the complexity above by exhibiting a smoothly embedded sphere generating $H_2(X_n(K))$ which intersects the cocore of the 2-handle the appropriate number of times. Let $S \subset X_n(P^c)$ denote the smoothly embedded sphere formed by capping off a slice disk for $P^c$ with the core of the attached 2-handle. It suffices to show that $d(S) \subset X_n(K)$ intersects the cocore, $C$, of the 2-handle of $X_n(K)$ exactly $\wrap(P)^c$ times. There is an obvious equality $\abs{d(S) \cap C} = \abs{S \cap d\invs(C)}$ which allows us to work in either $X_n(K)$ or $X_n(P^c)$ respectively; we will work in the latter. $\boundary C$ appears in the obvious Kirby diagram of $X_n(K)$ as a small meridian of $K$, moreover $C$ is given by taking the obvious disk which this meridian bounds in the diagram and pushing it into the 4-ball so it is disjoint from the attaching region of the 2-handle. It is an easy exercise to show that this surface is isotopic to the cocore of the 2-handle of the trace. We can describe $d\invs(C)$ by keeping track of this circle and disk throughout the dualization process described earlier. The key observation is that the dualization process exchanges the meridian of the pattern with the meridian of the solid torus in which the pattern lives. The disk which spans the meridian to the pattern and intersects it in one point is taken to the $\disk2$ factor of the solid torus, which intersects the dual pattern in exactly its wrapping number by definition. It is clear from its construction that $S$ will intersect this $\disk2$ exactly where the pattern $P^c$ does and thus 
            $$
                \abs{S \cap d\invs(C)} = \wrap(P^c) = \wrap(P)^c
            $$
            which proves the claim.
        \end{proof}

    \subsection{A Family of Dualizable Patterns}

    The patterns $P$ and $P\dual$ shown in \autoref{fig:PD1} and \autoref{fig:PandPdual1} have a striking similarity which is built into their construction. They live in a broader class of dualizable patterns which we construct in this section.

    \begin{figure}[!ht]
        \labellist
        \small\hair 2pt
         \pinlabel {$P$} [ ] at 29 153
         \pinlabel {$P\dual$} [ ] at 236 153
         \pinlabel \red{$\mu_V$} [ ] at 116 23
         \pinlabel \red{$\mu\sub{V\dual}$} [ ] at 340 25
        \endlabellist
        \centering
        \includegraphics{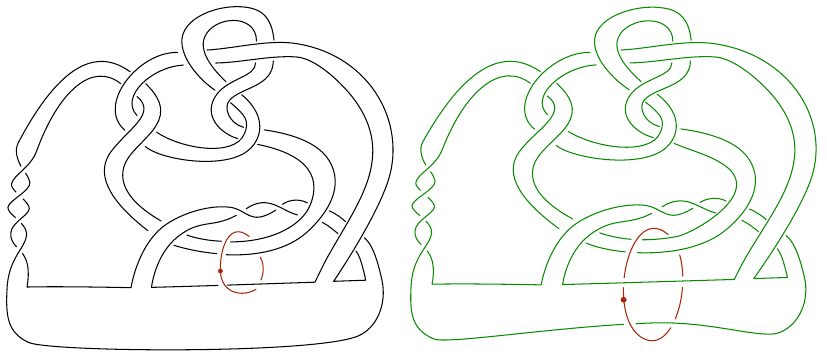} 
        \caption{The patterns $P$ and $P\dual$. Notice the dualization process reduces to changing which of the bottom strands is grabbed by the red curve.}
        \label{fig:PandPdual1}
    \end{figure}
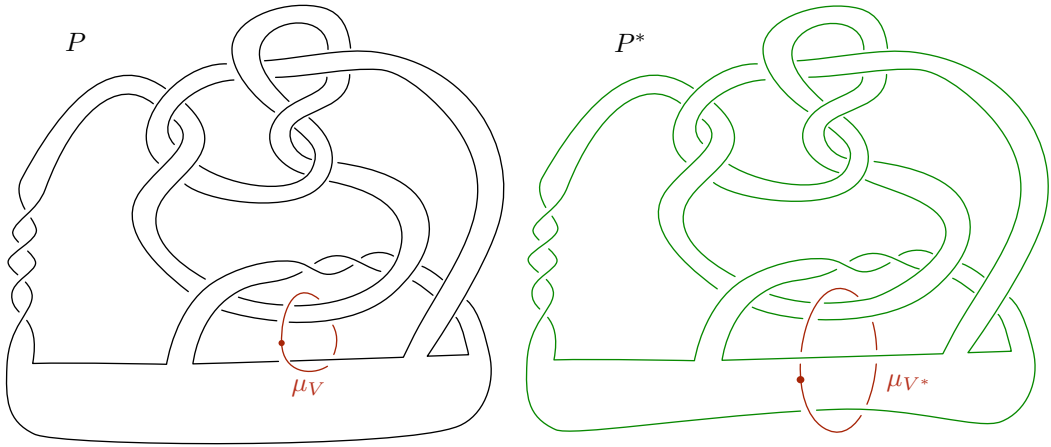

    \begin{proposition}
    \label{prop:mypatterns}
        Let $K \subset \sphere3$ be a knot and $U$ an unknot embedded in the complement of $K$. Assume $K$ intersects a spanning disk for $U$ exactly twice, the linking number $K \cdot U = 0$, and let some $n \in \Z$ be given. Associated to this data is a dualizable pattern $P$ which has wrapping number three, $g_3(P \subset V) = 1)$, and bears the same diagrammatic relationship to its dual as the example in \autoref{fig:PandPdual1}. Moreover $g_3(P\dual \subset V\dual) = 1$ and if $K$ is slice in either category and $n = 0$ then $P(U)$, and hence $P\dual(U)$, is slice in that category.  
    \end{proposition}

    \begin{proof}

    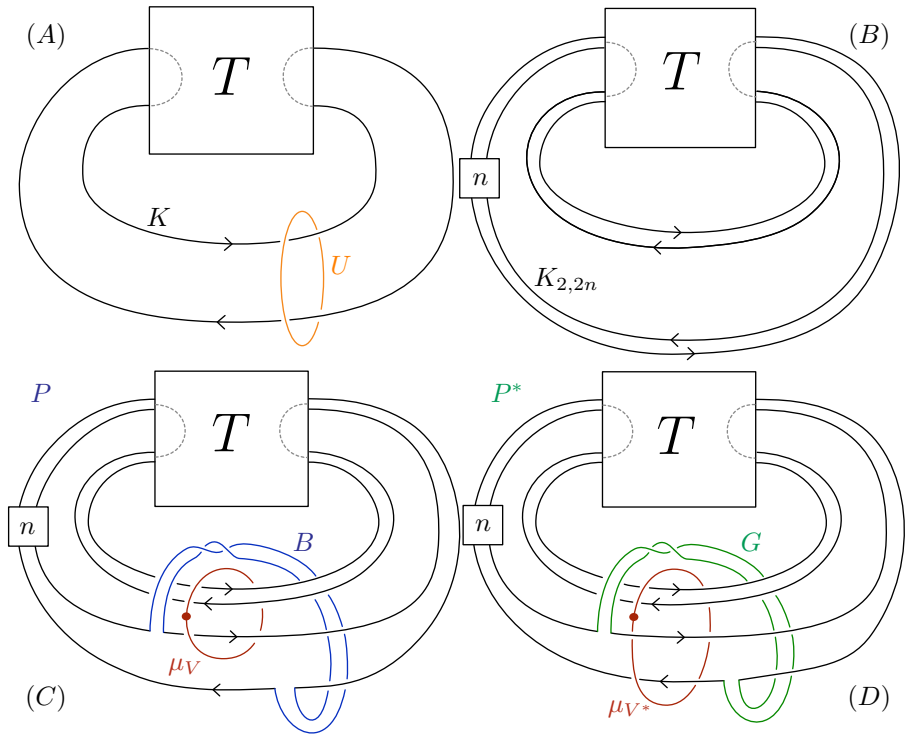
\begin{figure}
        \labellist
        \small\hair 2pt
         \pinlabel {$K$} [ ] at 57 199
         \pinlabel \huge{$T$} [ ] at 84 252
         \pinlabel \huge{$T$} [ ] at 252 254
         \pinlabel \huge{$T$} [ ] at 84 118
         \pinlabel \huge{$T$} [ ] at 250 117
         \pinlabel \orange{$U$} [ ] at 125 181
         \pinlabel {$n$} [ ] at 177 213
         \pinlabel {$n$} [ ] at 8 82
         \pinlabel {$n$} [ ] at 178 83
         \pinlabel \red{$\mu_V$} [ ] at 67 30
         \pinlabel \red{$\mu\sub{V\dual}$} [ ] at 233 14
         \pinlabel \blue{$B$} [ ] at 111 77
         \pinlabel \green{$G$} [ ] at 278 77
         \pinlabel {$K\sub{2,2n}$} [ ] at 209 175
         \pinlabel \blue{$P$} [ ] at 13 133
         \pinlabel \green{$P\dual$} [ ] at 187 133
         \pinlabel {$(A)$} [ ] at 15 267
         \pinlabel {$(B)$} [ ] at 322 267
         \pinlabel {$(C)$} [ ] at 15 18
         \pinlabel {$(D)$} [ ] at 322 18
        \endlabellist
        \centering
        \includegraphics{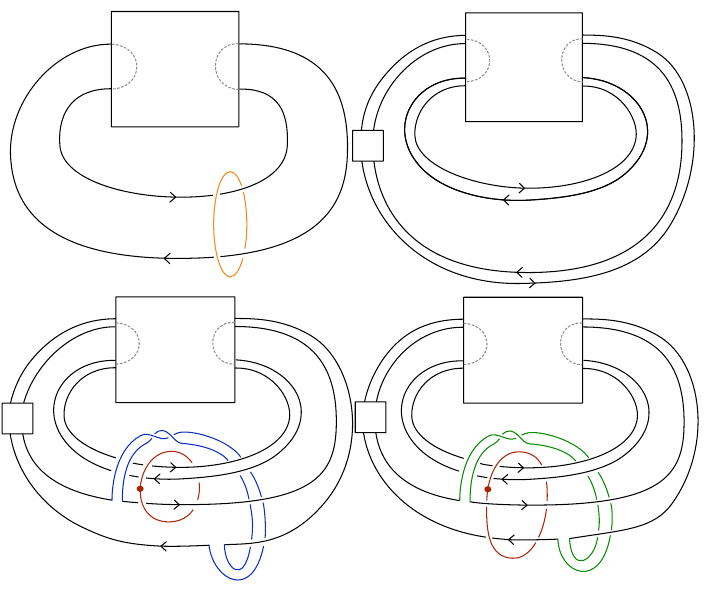}
        \caption{Constructing a dualizable pattern from the data specified in \autoref{prop:mypatterns}.}
        \label{fig:mypatterns}
    \end{figure}

    We start by explaining how to construct $P$ from this data. The unknot $U$ determines a decomposition of $K$ as the closure of a four-ended tangle $T$ as shown in \autoref{fig:mypatterns} (A). The fact that $K$ is a knot forces the tangle $T$ to have the connectivity shown by the two gray arcs inside the box, although of course the tangle may be far less trivial. We assume however that the tangle has been isotoped so that this closure gives a writhe-zero diagram of $K$. The first step in the construction is to pass to the $(2,2n)$-cable of $K$ with anti-coherent orientations, as shown in \autoref{fig:mypatterns} (B). The second step is to connect the two sides of the obvious annulus bounded by $K\sub{2,2n}$ by the band indicated in blue in (C) and green in (D) (it is the same band). This creates an obvious genus-one Seifert surface for the resulting knot. The final step is to remove a neighborhood of the red unknot either as in (C) to obtain the pattern $P$ or as in (D) to obtain its dual pattern $P\dual$. 

    We remark that the patterns in \autoref{fig:PandPdual1} correspond to the knot $6_1$ and the twisting parameter $n=0$. A close inspection of the \autoref{fig:PandPdual1} should make the position of the unknot clear. $6_1$ is well-known to be a smoothly slice knot, and the twisting parameter $n = 0$ gives the surface framing of the obvious copy of $6_1$ which lives on the natural genus-one Seifert surfaces for $P$ and $P\dual$. It follows that the knots $P(U)$ and $P\dual(U)$ are both smoothly slice. We can see explicit ribbon disks for each of them by performing a band move on the blue/green bands of \autoref{fig:mypatterns} which cuts each of the bands side-to-side. The bands then retract onto the black $(2,0)$-cable of $6_1$ which bounds a pair of disjoint parallel copies of the usual ribbon disk for $6_1$. 

    Returning to the proof of \autoref{prop:mypatterns}, we must show that the pattern constructed in \autoref{fig:mypatterns} (C) is dualizable with dual given by (D). The proof is given, in full generality by the computation in \autoref{fig:PD1} since none of the handle-slides or isotopies in that calculation interact with the tangle box. Thus the $(2,2)$-tangle coming from $6_1$ in that calculation could have been replaced with any tangle coming from any other knot, so long as it has the same connectivity (used in the middle-left sub-figure, when we pull the green component through the tangle). The final isotopy from the bottom-middle to bottom-right sub-figures also does not interact with the tangle except for one step in which a single left-handed crossing in the band passing though the tangle has to be pushed through the tangle and out the other side. Clearly this step can be achieved no matter what the tangle is, so long as the connectivity remains the same as the example in \autoref{fig:PD1}, and this is assumed in \autoref{prop:mypatterns}. 

    Finally we remark that the obvious Seifert surfaces for patterns and their duals constructed as in \autoref{prop:mypatterns} clearly each intersect the red curves representing $\mu_V$ and $\mu\sub{V\dual}$ in \autoref{fig:mypatterns} (C), (D) respectively exactly once. It follows that these surfaces, restricted to the complement of neighborhoods of the red curves, give embedded, genus-one, spanning surfaces for the two component links: $P \cup \lambda_V$ and $P\dual \cup \lambda\sub{V\dual}$ respectively. These surfaces, considered now in $\sphere3$, have obvious metabolizers if the knot $K$ used in the construction is smoothly slice and the twisting coefficient is zero. It follows that $P(U)$ and $P\dual(U)$ are smoothly slice under these hypotheses. 

    \end{proof}
    
\section{A Theorem about Classical Signatures}

    We digress for a moment in this section to prove \autoref{thm:2twists} about the classical signature. We will apply it in the next section to verify the conditions of \autoref{thm:highcxtyconstruction} for the pattern from \autoref{fig:PandPdual1} and $n$ odd. The case of $n$ even and non-zero will be handled by a different method. 

    The proof will proceed by using the method of Goeritz and Trotter, as described in \cite{gordon_signature_1978}, to calculate the classical signature of $K$ from a non-orientable spanning surface. We will show that there is a canonical way to convert a Seifert surface for $K$ into a non-orientable spanning surface for $K_{2}$. Moreover, the Goeritz form for $K_{2}$ can be computed from the Seifert form for $K$ and then diagonalized in full generality. The correction term can be easily computed and we will show that the correction term and change of signature offset each other, which proves the result.

    \subsection{Goeritz Forms and the Correction Term}

            Let $K$ be a knot and let $F$ be a spanning surface, possibly non-orientable, for $K$. Let $x_1, \ldots, x_n$ be an oriented basis for $H_1(F,\Z)$ and let $x^\tau_i$ denote the double cover of $x_i$ obtained by locally parametrizing the normal bundle as $F \times [-1,1]$ and lifting $x_i$ to $x_i \times \set{-1} \cup x_i \times \set{+1}$ in each chart. If $x_i$ is an orientation preserving loop in $F$ then $x_i^\tau$ is the disjoint union of two parallel copies of $x_i$, one on either side of $F$, and both oriented coherently with $x_i$. If $x_i$ is an orientation reversing loop then $x_i^\tau$ is a single circle which wraps twice around $x_i$ just as the boundary of a Mobius band wraps around its central circle, and which is oriented coherently with $x_i$. 
            
            We define the Goeritz matrix associated to the pair $(K,F)$ to be the form $\mathscr G_{ij} := lk(x_i,x_j^\tau)$. We observe that this form recovers the symmetrized Seifert form if $F$ is orientable since, 
            $$ 
                lk(x_i,x_j^\tau) = lk(x_i, x_j^+) + lk(x_i,x_j^-) = lk(x_i,x_j^+) + lk(x_j,x_i^+)
            $$
            For a proof that $\mathscr G$ is symmetric when $F$ is non-orientable see \cite{gordon_signature_1978}. 
            
            We define the correction term associated to $(K,F)$ as follows. Let $K_F$ denote a coherently oriented parallel copy of $K$ which is chosen to be disjoint from $F$, and observe that there is a unique such choice up to isotopies supported in a neighborhood of $K$. The correction term associated to $(K,F)$ is defined by $\eta := \half lk(K,K')$ which is always an integer. The central theorem of \cite{gordon_signature_1978} is:
            
            \begin{theorem}
            \label{thm:goeritz}
            Let $K$ be a knot in $\sphere3$ and let $F$ be a possibly non-orientable spanning surface for $K$ with Goeritz form $\mathscr G$ and correction term $\eta$. It follows that the classical knot signature of $K$ can be computed from $F$ by the formula: $$ \sigma(K) = \sign(\mathscr G) - \eta $$
            \end{theorem}
            
            We give the following example to demonstrate the practical utility of this theorem. The $(2,2n+1)$-torus knot has an obvious spanning surface homeomorphic to the Mobius band. The Goeritz form associated to this surface is the one-by-one form $[2n+1]$ and the correction term is also $2n+1$, therefore, 
            $$ 
                \sigma(T_{2,2n+1}) = \sign([2n+1]) - (2n+1) = 1 - 2n - 1 = \boxed{-2n} 
            $$ 
            The amazing brevity of this calculation will not be lost on those who have tried to compute the signature of torus knots using orientable spanning surfaces.
            
            The most practical way to approach this calculation for a specific knot $K$ is to perform an isotopy of the pair $(K,F)$ in $\sphere3$ to make $F$ look like a small disk in the plane with several bands attached. The bands will twist and link with each other for a generic knot, but we can arrange them to lie in the same plane as the disk, except when one of them passes over another, or when one twists around its core arc. We label the bands $B_1, \ldots, B_{2n}, N_1, \ldots, N_m$ so that the $B_*$ are the orientable bands and the $N_*$ are the non-orientable bands. Each band yields a generator of $H_1(F,\Z)$, obtained by orienting the core arc of the band and connecting the ends of the arc with a small, boundary-parallel arc in the disk. The orientation of the core arc can be chosen arbitrarily for the orientable bands, but for the non-orientable bands it is convenient to choose the orientation to be coherent with the orientation of the edges of the band (which both point in the same direction because the band is non-orientable!). These generators form a basis and we can easily compute the Goeritz form in this basis in terms of the writhes, linking numbers, and twisting numbers of the bands.
            
            \begin{definition}
            \label{def:bands}
            Let $F \subset \sphere3$ be a connected, possibly non-orientable surface with a single boundary component, $K$. Assume that $F$ is embedded as a flat disk with bands, $B_1, \ldots, B_{2n}, N_1, \ldots, N_m$, attached to it and assume that all bands come with oriented core arcs as in the previous paragraph. We give the following definitions for any bands $X_i,X_j$ of $F$:
                \begin{enumerate}[itemsep=1mm]
                    \item The \emph{writhe} of $X_i$, denoted $W(X_i)$, is twice the signed count of self-crossings of the band $X_i$ where the sign is determined by the sign of the induced self-crossing of the core arc of $X_i$.
                    \item The \emph{twisting number} of $X_i$, denoted $T(X_i)$ is the signed count of half twists in the band relative to the plane. The sign is determined by the handedness of the twisting as one travels along the core arc (note: this is independent of the orientation of the core arc!).  
                    \item The \emph{linking number} of $X_i$ with $X_j$, denoted $Lk(X_i,X_j)$, is the signed count of crossings between the bands $X_i$ and $X_j$ with the signs determined by the orientations of their cores. \smallskip
                    
                    \noindent The linking number of $X_i$ with itself is defined by $Lk(X_i,X_i) := W(X_i) + T(X_i)$. Note that this is twice what the reader might expect if we compare it to the standard definition of the linking number between knots. 
                \end{enumerate}
            \end{definition}
            
            \noindent These definitions yield the following:
            
        \begin{proposition}
            \label{prop:goeritzbands}
            In the setting of the previous definition, the Goeritz form and correction term associated to the pair $(K,F)$ can be computed by the following formulas:
            $$ \mathscr G = [Lk(X_i,X_j)]_{i,j = 1}^{2n+m} \hspace{2cm} \eta = \sum_{i,j = 1}^n Lk(N_i,N_j) $$
            In particular, one can compute $\eta$ directly from $\mathscr G$ by summing all the entries in $\mathscr G$ which correspond to a pair of non-orientable bands. 
        \end{proposition}
            
        \begin{proof}
            
        \begin{figure}[!ht]
            \labellist
            \small\hair 2pt
             \pinlabel \orange{$x_i$} [ ] at 113 23
             \pinlabel \orange{$x_j$} [ ] at 141 160
             \pinlabel {$K$} [ ] at 24 62
             \pinlabel \green{$x_j^\tau$} [ ] at 119 111
             \pinlabel \green{$x_i^\tau$} [ ] at 61 149
             \pinlabel {$(A)$} [ ] at 67 20
             \pinlabel {$(B)$} [ ] at 234 20
             \pinlabel {$(C)$} [ ] at 383 20
             \pinlabel \green{$x_i^\tau$} [ ] at 360 111
             \pinlabel \orange{$x_i$} [ ] at 324 73
             \pinlabel {$F$} [ ] at 273 84
             \pinlabel \orange{$x_i$} [ ] at 183 124
             \pinlabel \green{$x_j^\tau$} [ ] at 168 72
             \pinlabel \orange{$x_j$} [ ] at 268 105
            \endlabellist
            \centering
            \includegraphics{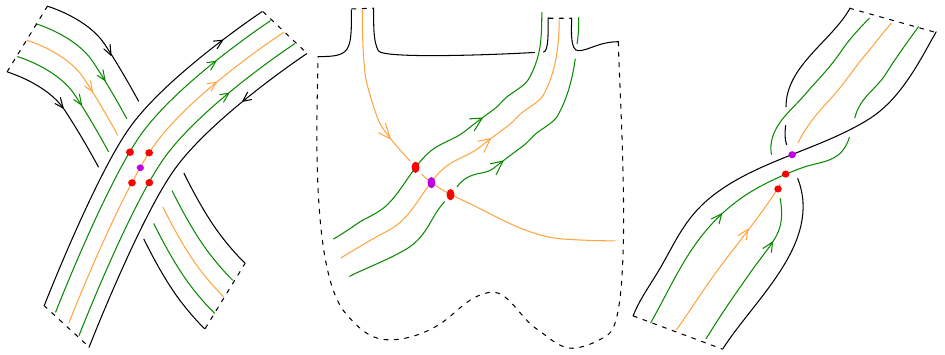}
            \caption{(A) bands crossing each other, (B) arcs and their push-offs meeting in the disk, (C) a band with a half twist. The cores of the bands are orange, the push-offs of the cores in green. Red dots mark green-orange crossings and purple dots mark orange-orange or black-black crossings respectively.}
            \label{fig:goeritzproof}
        \end{figure}
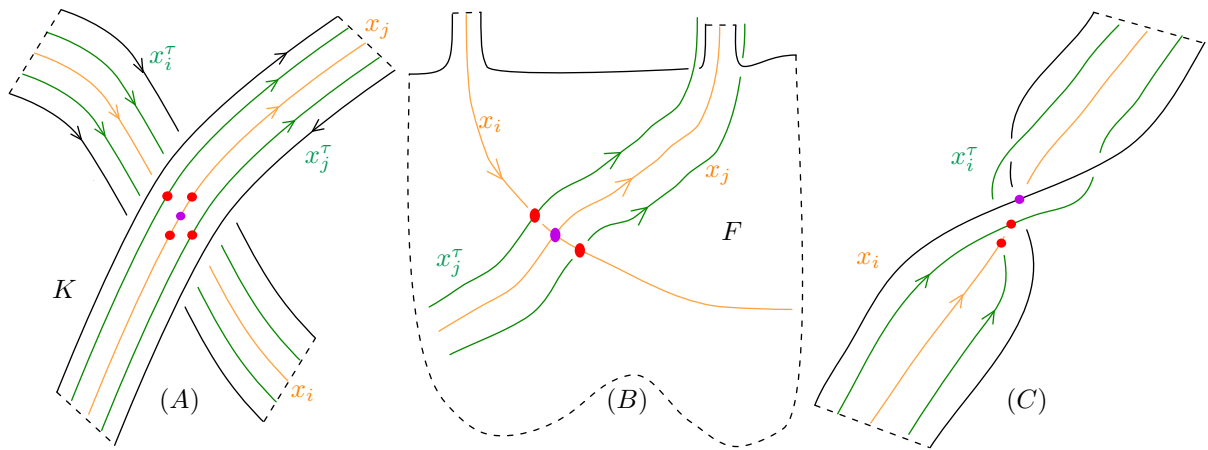    
            
        We begin by verifying the formula for the Goeritz form in the basis given by closing the oriented cores of the bands with small arcs in the disk. $\mathscr G_{ij} := lk(x_i,x_j^\tau)$ so we must count the crossings between these two links. There are three different families of crossings, which are shown in \autoref{fig:goeritzproof}. We check that the contributions which each of these local pictures make to $lk(x_i,x_j^\tau)$ and $Lk(X_i,X_i)$ are equal. In \autoref{fig:goeritzproof}, the generators corresponding to the cores of the bands are shown in orange and their push-offs in green. The crossings which contribute to $lk(x_i,x_j^\tau)$ are marked in red and those which contribute to $Lk(X_i,X_i)$ are in purple. 
            
        \begin{enumerate}[label=(\Alph*)]
            \item Two of the four red crossings contribute $\pm 1/2$ each to $lk(x_i,x_j^\tau)$, changing the sum by $\pm1$. The purple crossing also contributes $\pm 1$ to $Lk(X_i,X_j)$.
            \item The red crossings come with opposite sign so they contribute nothing to $lk(x_i,x_j^\tau)$. The purple crossing does not correspond to a band-over-band crossing so it contributes nothing to $Lk(X_i,X_j)$.
            \item The two red dots contribute $\pm 1/2$ each to $lk(x_i,x_i^\tau)$ and the purple crossing contributes $\pm 1$ to $W(X_i)$. 
        \end{enumerate}

        Next we analyze the correction term. Since $K'$ is disjoint from $F$, it follows that $K'$ is isotopic to the Seifert framed parallel $K''$ as $K'$ runs around the base disk and the orientable bands. This means the signed count of the crossings between $K$ and $K'$ will agree with the count between $K$ and $K''$ near the disk and orientable bands. This is zero since we can surger out the non-orientable bands from $(K,F)$ which makes $F$ an orientable spanning surface and $K'$ a disjoint parallel. It follows that $lk(K,K')$ can be computed entirely from counting crossings between $K'$ and the boundaries of the non-orientable bands. Where two non-orientable bands cross, the count changes by $\pm 4$ so the linking number changes by $\pm 2$ and the correction term by $\pm 1$. This is exactly the contribution the crossing makes to the linking number between the two bands. Wherever the band has a half twist, each of the strands of $K'$ must cross over the band which adds four crossings of the same sign. The count changes by $\pm 4$ so the correction term changes by $\pm 1$, see \autoref{fig:goeritzproof} (C). This is exactly the contribution which the half twist makes to the band's self-linking number. It follows that the correction term is exactly the sum of the entries of $\mathscr G$ corresponding to pairs of non-orientable bands. 

        \end{proof}

    \subsection{Adding Pairs of Twists}

     We will need the following lemma:

        \begin{lemma}
        \label{lem:niceform}
        Let $K \subset Y^3$ be a null-homologous knot in an oriented 3-manifold. Let $D \subset Y^3$ be an embedded disk with oriented boundary $U$ such that $U \cdot K = \pm 1$ and therefore,
        $$
            \abs{K \cap D} = 1 + 2c
        $$
        for some $c \in \Zpos$. It follows that there exists an orientable spanning surface $F \subset Y^3$ for $K$ such that $F \cap D$ consists of $c$ arcs on $D$ connecting oppositely oriented intersection points between $K$ and $D$ and one arc connecting the remaining positive intersection point to a point on $U = \boundary D$. Thus the intersection of $F$ with a neighborhood of $D$ looks like the upper left of \autoref{fig:spanningsurface}.
        \end{lemma}
        
        \begin{figure}[ht]
            \labellist
            \small\hair 2pt
             \pinlabel {$F_{1+2c}$} [ ] at 78 30
             \pinlabel {$F$} [ ] at 306 30
            \endlabellist
            \centering
            \includegraphics{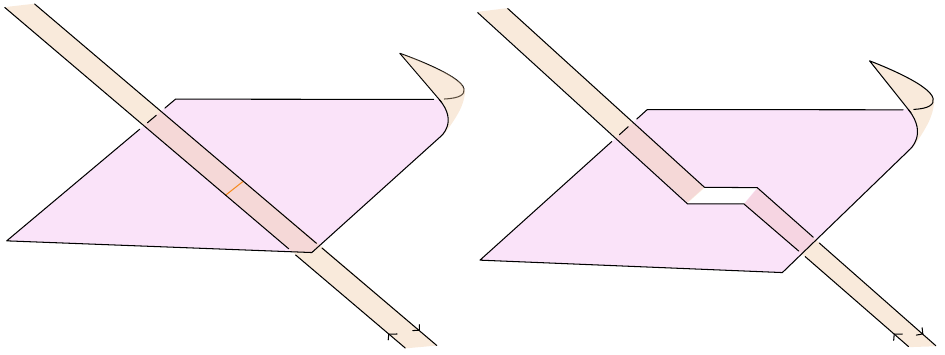}  
            \caption{Resolving a ribbon intersection of a Seifert Surface.}
            \label{fig:ribbonresolution}
        \end{figure}
        
        \begin{proof}
        We can pair up the strands of $K$ which intersect $D$ into oppositely oriented pairs, with one strand left over. Each oppositely oriented pair looks locally like the boundary of a band passing through $D$ and we can ensure that each resulting arc of intersection is disjoint. We can modify $K$ by surgering each of these bands along its arc of intersection with $D$, which transforms the $K$ into a $(c+1)$-component oriented link $L_1$. There is one point of intersection left between $L_1$ and $D$. We can remove it by doing a finger move along an arc on $D$ which avoids the previous arcs of intersection between $D$ and the bands. This gives us an oriented link $L_0 \subset Y^3 \setminus \nbhd D$. Clearly $L_0$ is also null-homologous in $Y^3\setminus \nbhd D$ and it follows that $L_0$ has a connected, oriented spanning surface $F_0 \subset Y^3 \setminus \nbhd{D}$. We can recover $K$ from $L_0$ and simultaneously convert $F_0$ into an immersed spanning surface for $K$ by attaching to $F_0$ the obvious ribbons which undo the surgeries and the finger move we used to produce $L_0$. Let the new, immersed, spanning surface for $K$ which we obtain be called $F\sub{1+2c}$. It follows that the only self-intersections of $F\sub{1+2c}$ with itself are ribbon intersections. These can be resolved according to the local model shown in \autoref{fig:ribbonresolution}. Let $F$ be the result of resolving the self-intersections of $F_{1+2c}$ in this manner. It now follows that $F$ satisfies the conclusion of the lemma.
        \end{proof}
        
        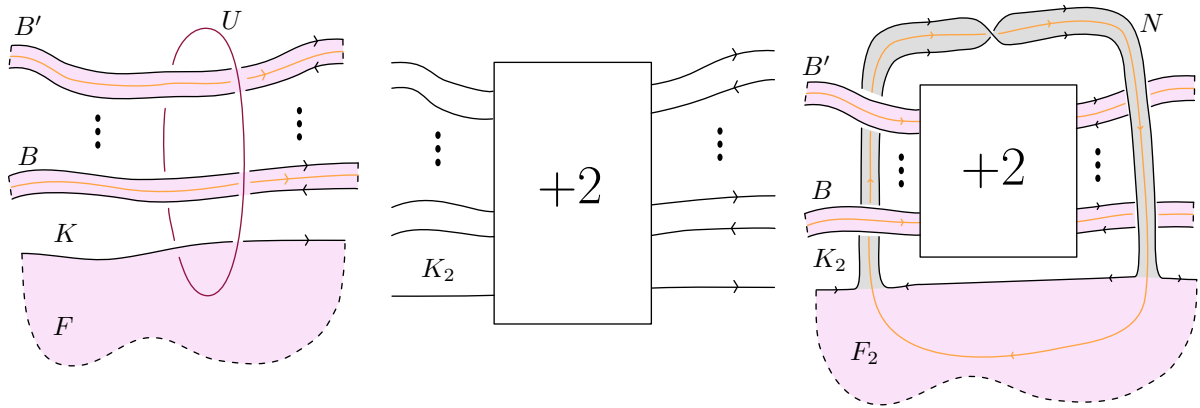
\begin{figure}[ht]
            \labellist
            \small\hair 2pt
             \pinlabel {$U$} [ ] at 91 158
             \pinlabel \huge{$+2$} [ ] at 217 95
             \pinlabel \huge{$+2$} [ ] at 375 102
             \pinlabel {$N$} [ ] at 434 157
             \pinlabel {$F_2$} [ ] at 327 33
             \pinlabel {$F$} [ ] at 28 43
             \pinlabel {$K$} [ ] at 29 78
             \pinlabel {$K_2$} [ ] at 168 64
             \pinlabel {$K_2$} [ ] at 314 68
             \pinlabel {$B$} [ ] at 15 107
             \pinlabel {$B'$} [ ] at 15 156
             \pinlabel {$B$} [ ] at 311 94
             \pinlabel {$B'$} [ ] at 309 141
            \endlabellist
            \centering
            \includegraphics{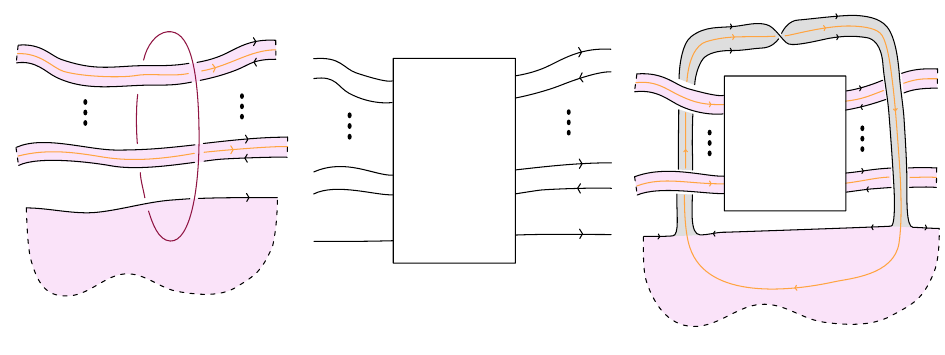}
            \caption{Adding two right-handed twists to $K$ and modifying its spanning surface.}
            \label{fig:spanningsurface}
        \end{figure}

        We can now give the proof of \autoref{thm:2twists}:

        \begin{proof}[Proof of \autoref{thm:2twists}]
            
        Consider a local neighborhood of the spanning disk for $U$ and its intersection with $K$. We see $U$ as a standard unknot and $2n+1$ strands of $K$ passing through it, $n+1$ passing positively and $n$ passing negatively. We can pair up each positive strand with a negative strand and we will have one positive left over. We can assume without loss of generality that there is a Seifert surface $F$ for $K$ such that each oppositely oriented pair of strands passing through $U$ is the boundary of an orientable band of $F$ by the lemma. We can further assume that the intersection between the disk and $F$ is an arc connecting the left-over strand to $U$ and several disjoint intervals corresponding to the bands. Now consider $K_{2}$. We can see from \autoref{fig:spanningsurface} that $K_{2}$ is equal to $K$ outside a neighborhood of the spanning disk for $U$, and moreover there is a natural way to modify $F$ to obtain a non-orientable spanning surface for $K_{2}$ which we will call $F_{2}$. Notice that the orientable bands of $F_{2}$ correspond to the bands of $F$ in an obvious way, while exactly one non-orientable band has been added.
        
        We will compare the Goeritz form $\sG$ of $F$ to the form $\sG_{2}$ of $F_{2}$. Let's order the bands of $F$ by putting the ones which don't pass through the spanning disk for $U$ first then those that do. If $B, B'$ are a bands with self-linking numbers $b, b'$ which pass through the disk $l, l'$ times, counted with sign, then the self-linking of the corresponding bands of $F_{2}$ will be $b+4l^2, b+{4l'}^2$ respectively. If $Lk(B,B') = p$ then the linking number of the corresponding bands in $F_{2}$ will be $p + 4ll'$. From the diagram, we can see that the self-linking of the new, non-orientable band, $N$, is $+1$ and of course its linking numbers with $B, B'$ are $2l,2l'$. Thus the intersection of the three rows and columns corresponding to $B, B', N$ form the sub-matrix given on the left below. We can diagonalize this sub-matrix using the pairs of row and column operations indicated by the arrows below. 
        $$
        \begin{pmatrix}
        b + 4l^2 & p + 4ll' & 2l \\
        p + 4l'l & b' + {4l'}^2 & 2l' \\
        2l & 2l' & 1
        \end{pmatrix}
        \buildrel{((2) - 2l' \cdot (3)) \mapsto (2)}\over\leadsto
        \begin{pmatrix}
        b + 4l^2 & p & 2l \\
        p & b' & 0 \\
        2l & 0 & +1
        \end{pmatrix}
        \buildrel{((1) - 2l \cdot (3)) \mapsto (1)}\over\leadsto
        \begin{pmatrix}
        b & p & 0 \\
        p & b' & 0 \\
        0 & 0 & 1
        \end{pmatrix}
        $$
        Since $B,B'$ could have been any two bands from $F$ we see that this sequence of operations applied to each disjoint pair of corresponding bands in $F_{2}$ gives a matrix congruence from $\sG_{2}$ to $\sG \oplus [+1]$. It follows that $\sign(\sG_{2}) = \sign(\sG)+1$. Since there is only one non-orientable band on $F_{2}$, we see that the correction term is its self-intersection which gives $\eta = +1$. Thus, 
        $$ 
            \sigma(K_{2}) = \sign(\sG_{2}) - \eta = \sign(\sG) + 1 - 1 = \sign(\sG) = \sigma(K)
        $$
        The result now follows from the obvious inductive argument on $n$. 
        
        \end{proof}

\section{The Examples; Proof of the Main Theorem}

    \subsection{The Diagrams}

    We now begin the process of verifying the hypotheses of \autoref{thm:highcxtyconstruction} for the two patterns shown in \autoref{fig:PandPdual}, with the Levine-Tristram signature function playing the role of the compatible invariant.  We remark that $\sigma(-)$ denotes the classical knot signature, which is the same as the Levine-Tristram signature $\sigma(-,t = -1)$ which is always valued in $2\Z$. 

    \begin{proposition}
    \label{prop:oddcase}
    Let $P$ be the pattern described in \autoref{fig:PandPdual}. For each $k \in \Z$, 
        $$ \half\sigma(P_{2k}) = 0, \qquad \half\sigma(P\sub{1+2k}) = 1 $$
    It follows that the knots $K\sub{n,c}$ constructed as in \autoref{thm:highcxtyconstruction} satisfy the conclusions of \autoref{thm:main} for all odd integers $n$. 
    \end{proposition}
    \begin{figure}[!ht]
        \labellist
        \small\hair 2pt
         \pinlabel {$P$} [ ] at 29 153
         \pinlabel {$P\dual$} [ ] at 236 153
         \pinlabel \red{$\mu_V$} [ ] at 116 23
         \pinlabel \red{$\mu\sub{V\dual}$} [ ] at 340 25
        \endlabellist
        \centering
        \includegraphics{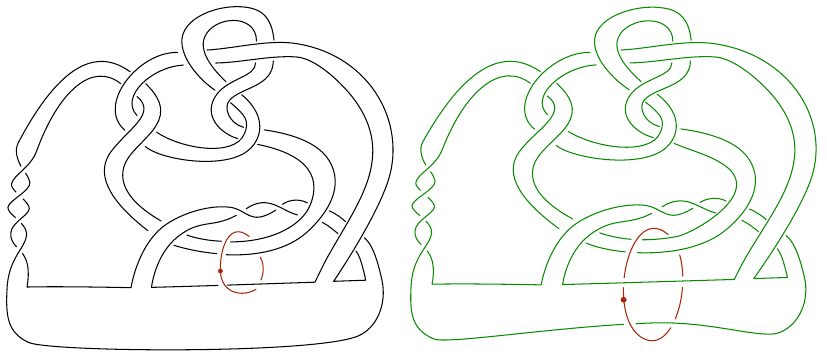} 
        \caption{The patterns $P$ and $P\dual$.}
        \label{fig:PandPdual}
    \end{figure}    
    \begin{proof}
    We apply \autoref{thm:2twists} to obtain: 
    $$
        \sigma(P_{2k}) = \sigma(P) = 0 \hspace{1cm} \sigma(P_{1+2k}) = \sigma(P_1) 
    $$
    and so it will suffice to compute $\half\sigma(P_1)$. We do this explicitly by drawing a Seifert surface for $P_1$ then writing down the associated Seifert form $A_1$ and computing the signature:
    $$
        A_1 = \tiny{\begin{pmatrix}
          2 & 0 & -1 & 1 \\
         -1 & 1 &  0 & 0 \\
         -1 & 0 &  0 & 0 \\
          0 & 0 &  1 & 0 \\
        \end{pmatrix}},
        \hspace{1cm}
        \half\sigma(P_1) = \half\sigma(A_1 + A^T_1) = \half\sigma
            \tiny{\begin{pmatrix}
              4 & -1 & -2 & 1 \\
             -1 &  2 &  0 & 0 \\
             -2 &  0 &  0 & 1 \\
              1 &  0 &  1 & 0 \\
            \end{pmatrix}}
        = 1
    $$
    We can see from inspecting the diagram that $\wrap(P\dual)=3$ and $g_3(P\dual\subset V\dual) = 1$. It follows from these values and the computation above that the bounds obtained from \autoref{thm:highcxtyconstruction} agree with the conclusions of \autoref{thm:main} for the associated knots $K\sub{n,c}$ assuming $n$ is odd.  
    \end{proof}

    This gives smoothly $n$-shake-slice knots with any smooth or topological 4-genus and arbitrarily high complexity for any odd $n$. We need the greater generality of the Levine-Tristram signatures to tackle even $n$, since our computation above shows $\sigma(P\sub{-2k}) = 0$. The main thrust of our analysis will be the following: we will construct a family of Seifert surfaces $F_n$ for $P_n$ which yield Seifert forms $A_n$. We will be able to give a general formula for the matrix $A_n$ and deduce a general formula for the symmetrized Alexander polynomial $\Delta_{n}(t)$ of the knot $P_n$. The Levine-Tristram signature is known to change exactly at those complex numbers on the unit circle where the symmetrized Alexander polynomial passes through zero. We will show that for any $n \in \Z$, $\Delta_n(t)$ attains both positive and negative values on $\sphere1 \subset \C$. It then follows from the intermediate value theorem that there is some $t_n \in \sphere1$ which depends on $n$ where $\half\sigma(P_n,t_n) \ne 0$. Since $P_n$ bounds an obvious immersed, ribbon, genus-one surface in $\sphere3$, it follows that $\half\sigma(P_n,t_n) = \pm 1$ and the proof of \autoref{thm:main} follows exactly as the proof of \autoref{prop:oddcase} does above. 

    \subsection{Adding Twists around Algebraically One Strand}\hfill\newline
        
        We will now explain how to create a family of Seifert surfaces for $P_n$ starting from a given Seifert surface for $P$. We use the following extension of \autoref{lem:niceform}. 

        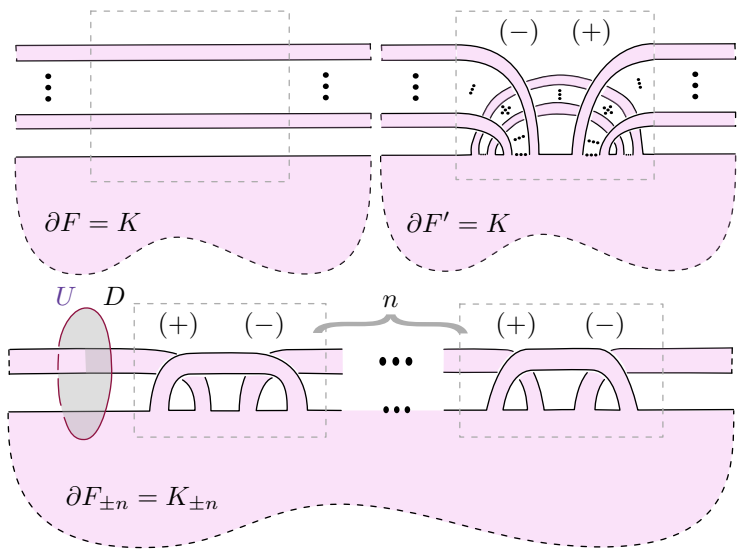
\begin{figure}[ht]
            \labellist
            \small\hair 2pt
             \pinlabel {$\boundary F = K$} [l] at 15 130
             \pinlabel {$\boundary F' = K$} [l] at 150 130
             \pinlabel {$(-)$} [ ] at 194 201
             \pinlabel {$(+)$} [ ] at 221 201
             \pinlabel {$(+)$} [ ] at 67 91
             \pinlabel {$(-)$} [ ] at 99 91
             \pinlabel {$n$} [ ] at 146 100
             \pinlabel {$(+)$} [ ] at 193 91
             \pinlabel {$(-)$} [ ] at 226 91
             \pinlabel {$D$} [ ] at 43 102
             \pinlabel \purple{$U$} [ ] at 25 102
             \pinlabel {$\boundary F\sub{\pm n} = K\sub{\pm n}$} [l] at 23 26
            \endlabellist
            \centering
            \includegraphics{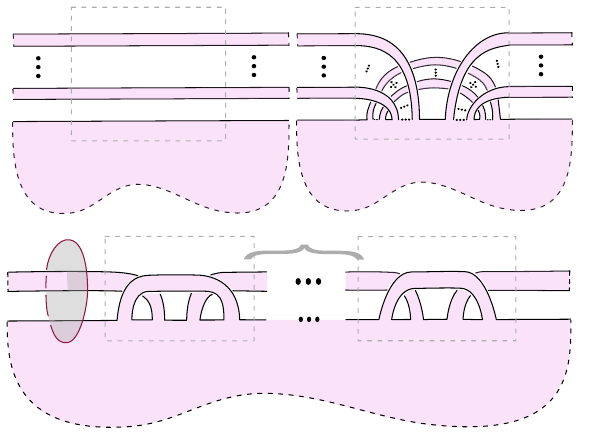}
            \caption{Modifying a Seifert surface to add twists. Switching all the band-over-band crossings on the $(\pm)$ side of a gray box adds a $\pm1$ full twist to $K$. Concatenation creates a Seifert surface $F\sub{\pm n}$ for $K_{\pm n}$.}
            \label{fig:Seifertmodification}
        \end{figure}
        
        \begin{lemma}
        \label{lem:magicarches}
        Let $K \cup U \subset \sphere3$ be as in \autoref{lem:niceform}, and let $K_n$ be as in \autoref{thm:2twists}. Let $F$ be the Seifert surface for $K$ obtained from \autoref{lem:niceform}. We may convert the surface $F$ into a new Seifert surface $F_n$ for $K_n$ by making a finite number of local modifications to $F$ as indicated in \autoref{fig:Seifertmodification}.
        \end{lemma}
        
        The proof of the lemma is obvious from the figure up to one small point. Since we are encoding the two sides of the Seifert surface by color (Indigo/Orange), the modification indicated in \autoref{fig:Seifertmodification} only applies when all the bands passing through $U$ have the same side facing out of the page as the portion of the surface which intersects $U$. This can always be achieved by putting a Reidemeister II move in each band which has the wrong side facing outwards (one of the new crossings goes on each side of the disk $D$ then we shrink the neighborhood of $D$ so that neither crossing appears). 
        
        Applying the lemma to an appropriate Seifert surface for $P$ yields surfaces for $P_n$. These surfaces are shown in \autoref{fig:Pseifert} and we call them $F_n$. We connect the ends of the cores of the bands by arcs in the central disk of $F_n$ to obtain a basis for $H_1(F_n,\Z)$ which we label $(1), \ldots, (2n+2)$. We orient this basis by orienting all the arcs lying in the central disk of $F_n$ from left to right. The associated Seifert form, $A_n := [lk((i),(j)^+)]$, is a $(2n+2)\times(2n+2)$ matrix for $\abs{n} \ge 1$. We see by inspecting the diagram that $A_n$ and $A\sub{-n}$ have the following forms:
        
        $$
        A_n = \tiny{\left(
                    \begin{array}{ccc|ccccc}
                          2 &   0    & -1  & 1 &   0     &        & 0 & 0 \\
                         -1 &   1    &  0  & 0 &   0     & \cdots & 0 & 0 \\
                         -1 &   0    &  0  & 0 &   0     &        & 0 & 0 \\
                                            \hline
                          0 &   0    &  0  & 0 &   1     &        & 0 & 0 \\
                          0 &   0    &  0  & 0 &   0     & \ddots & 0 & 0 \\
                            & \vdots &     &   & \ddots  & \ddots & 1 & 0 \\
                          0 &   0    &  0  & 0 &   0     &    0   & 0 & 1 \\
                          0 &   0    &  1  & 0 &   0     &    0   & 0 & 0
                        \end{array}
                \right)}, \qquad
        A\sub{-n} = \tiny{\left(
                    \begin{array}{ccc|ccccc}
                          2 &   0    & -1  & 0 &   0     &        & 0 & 0 \\
                         -1 &   1    &  0  & 0 &   0     & \cdots & 0 & 0 \\
                         -1 &   0    &  0  & 0 &   0     &        & 0 & -1 \\
                                            \hline
                         -1 &   0    &  0  & 0 &   0     &        & 0 & 0 \\
                          0 &   0    &  0  &-1 &   0     & \ddots & 0 & 0 \\
                            & \vdots &     &   & \ddots  & \ddots & 0 & 0 \\
                          0 &   0    &  0  & 0 &   0     &   -1   & 0 & 0 \\
                          0 &   0    &  0  & 0 &   0     &    0   &-1 & 0
                        \end{array}
                \right)}
        $$
        
        \noindent We describe $A_n$ in terms of the four sub-matrices indicated in the figure:
        \begin{itemize}[label=---, itemsep=1mm]
            \item The lower-right block of $A_n$ is a $(2n-3)\times(2n-3)$ matrix consisting of ones down the $(i,i+1)$-diagonal and zeroes elsewhere.
            \item The top-right and bottom-left blocks consist of a single $1$ in the top-left and bottom-right entries respectively.
            \item The top-left block is independent of $n$ (and given in the example above). 
        \end{itemize}
        This description of $A_n$ can be checked directly from the figure for any value of $n \in \N$. We can obtain $A\sub{-n}$ by taking the negative transpose of the all but the upper-left block of $A_n$. We will use $A_n$ to compute the symmetrized Alexander polynomial and Levine-Tristram signature function associated to the knots $P_n$. We review the relationship between these two invariants.
        
        \begin{figure}[!hb]
            \labellist
            \small\hair 2pt
             \pinlabel {$n$} [ ] at 170 77
             \pinlabel {$\boundary F_n = P_n$} [ ] at 60 15
             \pinlabel \green{$>$} [ ] at 60 28.25
             \pinlabel \green{$>$} [ ] at 165 17.5
             \pinlabel \green{$>$} [ ] at 120 33
             \pinlabel \green{$>$} [ ] at 147 33
             \pinlabel \green{$>$} [ ] at 189 32.9
             \pinlabel \rotatebox{-15}{\green{$>$}} [ ] at 212 29.75
             \pinlabel \rotatebox{45}{\green{$>$}} [ ] at 237 31.5
             \pinlabel \tiny\green{$(1)$} [ ] at 28 23
             \pinlabel \tiny\green{$(2)$} [ ] at 115 15
             \pinlabel \tiny\green{$(4)$} [ ] at 128 26
             \pinlabel \tiny\green{$(5)$} [ ] at 149 26
             \pinlabel \tiny\green{$(2n+1)$} [r] at 190 26
             \pinlabel \tiny\green{$(2n+2)$} [ ] at 204 24
             \pinlabel \tiny\green{$(3)$} [ ] at 248 30
            \endlabellist
            \centering
            \includegraphics[scale=1.2]{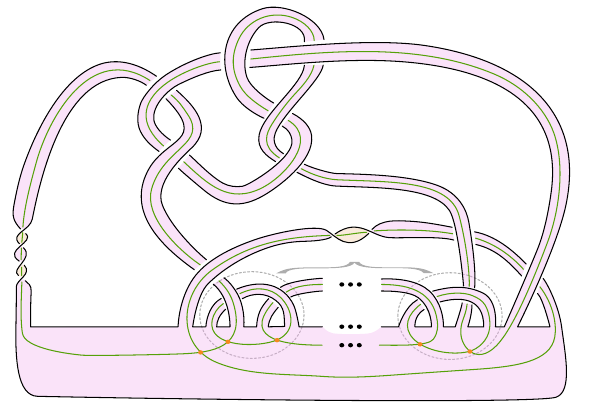}
            \caption{A Seifert surface, $F_n$, for $P_n$ with a labeled basis, $(1), \ldots, (2n+2)$, for $H_1(F_n,\Z)$.}
            \label{fig:Pseifert}
        \end{figure}
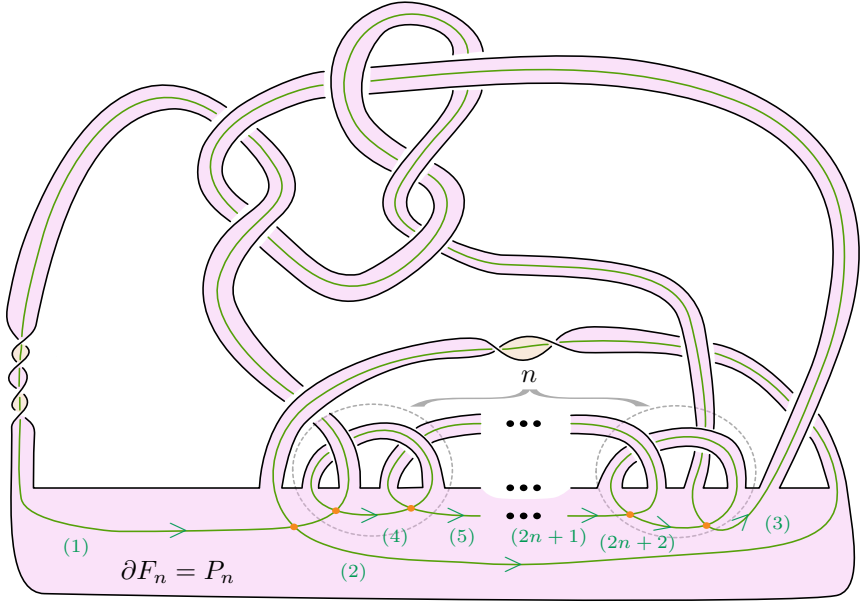

        \begin{proposition}[\cite{gilmer_signature_2015}]
        \label{prop:alexLT}
        Let $K$ be a knot and $A$ be a Seifert form for $K$. The symmetrized Alexander polynomial of $K$ is given by: 
        $$ 
            \Delta_K(t) = t^{-\dim(A)/2}\det(tA-A^T) = \det(\sqrt{t}A-\frac{1}{\sqrt{t}}A^T)
        $$
        \hfill and $\Delta_K(t) \in \R$ for any $t \in \C$ satisfying $\abs{t} = 1$. \medskip
        
        \noindent The Levine-Tristram signature of $K$ at a unit-norm complex number $t$ is given by:
        $$
            \sigma(K,t) = \sigma((1-t)A + (1-\frac{1}{t})A^T)
        $$
        Moreover:
        $$
            (t^{-1/2} - t^{1/2})^{\dim(A)} \Delta_K(t) = \det((1-t)A + (1-\frac{1}{t})A^T)
        $$
        from which it follows that:
        \begin{itemize}
            \item $\sigma(K,-): S^1(\subset \C) \to \Z$ is constant where $\Delta_K(t) \ne 0$
            \item $\sigma(K,-)$ must change wherever the sign of $\Delta_K(t)$ changes on $S^1 \subset \C$
        \end{itemize}
        \end{proposition}
        
        We exploit these properties to prove:
        
        \begin{proposition}
        \label{prop:salexformula}
        Let $n \in \Z$ and let $\Delta_n(t)$ denote the symmetrized Alexander polynomial of $P_n$. It is given by the formula:
        $$
            \Delta_n(t) = -\frac{1}{t^{n+1}} + \frac{2}{t^{n}} - \frac{1}{t^{n-1}} + \frac{2}{t} - 3 + 2t - t^{n-1} + 2t^{n} - t^{n+1}
        $$
        Clearly $\Delta_n(1) = 1$, but for $t^{\abs{n}} = -1$ and $\Re(t)< 7/8$ we obtain $\Delta_n(t) < 0$. It follows that for each $n \ne 0$ there exists some root of unity $t_* \in \sphere1 \subset \C$ where $\sigma(P_{n},t_*) \ne 0$.
        \end{proposition}

        \begin{proof}
        
        The proof splits into three cases: $n=0$, for which one can check explicitly using the untwisted genus-one Seifert surface for $P(U)$ that $\Delta_P(t) = 1$, agreeing with the formula in the proposition. The remaining cases $n < 0$ and $n > 0$ turn out the same since $\Delta_{-n}(t) = \Delta_{n}(t)$. We will verify the formula for $n>0$ and see along the way why it is invariant under negating $n$. 
        
        We will reduce $tA_n - A^T$ into the block sum of a 4-by-4 matrix depending on $n$ and the matrix $-\one_{2n-2}$. This process consists of a sequence of elementary row and column operations followed by cyclically permuting an odd number of rows, all of which preserves the determinant. We describe this process by an example, and claim that it is clear from this example that the process can be completed for any $n \ne 0$. Let $n\ge 0$ be given ($n=6$ in the example) and let:
        \begin{align*}
            B_n&:= tA_n - A_n^T = \tiny{\left(
                                            \begin{array}{ccc|ccccccccccc}
                                             2 t-2 & 1 & 1-t & t & 0 & 0 & 0 & 0 & 0 & 0 & 0 & 0 & 0 & 0 \\
                                             -t & t-1 & 0 & 0 & 0 & 0 & 0 & 0 & 0 & 0 & 0 & 0 & 0 & 0 \\
                                             1-t & 0 & 0 & 0 & 0 & 0 & 0 & 0 & 0 & 0 & 0 & 0 & 0 & -1 \\
                                             \hline
                                             -1 & 0 & 0 & 0 & t & 0 & 0 & 0 & 0 & 0 & 0 & 0 & 0 & 0 \\
                                             0 & 0 & 0 & -1 & 0 & t & 0 & 0 & 0 & 0 & 0 & 0 & 0 & 0 \\
                                             0 & 0 & 0 & 0 & -1 & 0 & t & 0 & 0 & 0 & 0 & 0 & 0 & 0 \\
                                             0 & 0 & 0 & 0 & 0 & -1 & 0 & t & 0 & 0 & 0 & 0 & 0 & 0 \\
                                             0 & 0 & 0 & 0 & 0 & 0 & -1 & 0 & t & 0 & 0 & 0 & 0 & 0 \\
                                             0 & 0 & 0 & 0 & 0 & 0 & 0 & -1 & 0 & t & 0 & 0 & 0 & 0 \\
                                             0 & 0 & 0 & 0 & 0 & 0 & 0 & 0 & -1 & 0 & t & 0 & 0 & 0 \\
                                             0 & 0 & 0 & 0 & 0 & 0 & 0 & 0 & 0 & -1 & 0 & t & 0 & 0 \\
                                             0 & 0 & 0 & 0 & 0 & 0 & 0 & 0 & 0 & 0 & -1 & 0 & t & 0 \\
                                             0 & 0 & 0 & 0 & 0 & 0 & 0 & 0 & 0 & 0 & 0 & -1 & 0 & t \\
                                             0 & 0 & t & 0 & 0 & 0 & 0 & 0 & 0 & 0 & 0 & 0 & -1 & 0 \\
                                            \end{array}
                                            \right)}
        \end{align*}
        We note that $B_n$ is a $(2n+2)\times(2n+2)$-matrix. Observe that every column in the lower-right block, except the rightmost, contains a $-1$ and a $t$. We can perform elementary row operations on $B_n$ using these $-1$ entries to kill the $t$ entries in their respective columns. After this process we obtain the matrix $B_n'$ given by:
        $$
            B'_n := \tiny{\left(
                    \begin{array}{ccc|ccccccccccc}
                     2 t-2 & 1 & 1-t & t & 0 & 0 & 0 & 0 & 0 & 0 & 0 & 0 & 0 & 0 \\
                     -t & t-1 & 0 & 0 & 0 & 0 & 0 & 0 & 0 & 0 & 0 & 0 & 0 & 0 \\
                     1-t & 0 & 0 & 0 & 0 & 0 & 0 & 0 & 0 & 0 & 0 & 0 & 0 & -1 \\
                     \hline
                     -1 & 0 & t^6 & 0 & 0 & 0 & 0 & 0 & 0 & 0 & 0 & 0 & 0 & 0 \\
                     0 & 0 & 0 & -1 & 0 & 0 & 0 & 0 & 0 & 0 & 0 & 0 & 0 & t^5 \\
                     0 & 0 & t^5 & 0 & -1 & 0 & 0 & 0 & 0 & 0 & 0 & 0 & 0 & 0 \\
                     0 & 0 & 0 & 0 & 0 & -1 & 0 & 0 & 0 & 0 & 0 & 0 & 0 & t^4 \\
                     0 & 0 & t^4 & 0 & 0 & 0 & -1 & 0 & 0 & 0 & 0 & 0 & 0 & 0 \\
                     0 & 0 & 0 & 0 & 0 & 0 & 0 & -1 & 0 & 0 & 0 & 0 & 0 & t^3 \\
                     0 & 0 & t^3 & 0 & 0 & 0 & 0 & 0 & -1 & 0 & 0 & 0 & 0 & 0 \\
                     0 & 0 & 0 & 0 & 0 & 0 & 0 & 0 & 0 & -1 & 0 & 0 & 0 & t^2 \\
                     0 & 0 & t^2 & 0 & 0 & 0 & 0 & 0 & 0 & 0 & -1 & 0 & 0 & 0 \\
                     0 & 0 & 0 & 0 & 0 & 0 & 0 & 0 & 0 & 0 & 0 & -1 & 0 & t \\
                     0 & 0 & t & 0 & 0 & 0 & 0 & 0 & 0 & 0 & 0 & 0 & -1 & 0 \\
                    \end{array}
                    \right)}
        $$
        Rows (5) to (2n+2) now contain a $-1$ entry so we can use them to clear their respective rows. We then use the $-1$ in the (5,4) position to kill the $t$ in the (1,4) position:
        $$
        B\pp_n := \tiny{\left(
                    \begin{array}{ccc|ccccccccccc}
                     2 t-2 & 1 & 1-t & 0 & 0 & 0 & 0 & 0 & 0 & 0 & 0 & 0 & 0 & t^6 \\
                     -t & t-1 & 0 & 0 & 0 & 0 & 0 & 0 & 0 & 0 & 0 & 0 & 0 & 0 \\
                     1-t & 0 & 0 & 0 & 0 & 0 & 0 & 0 & 0 & 0 & 0 & 0 & 0 & -1 \\
                     \hline
                     -1 & 0 & t^6 & 0 & 0 & 0 & 0 & 0 & 0 & 0 & 0 & 0 & 0 & 0 \\
                     0 & 0 & 0 & -1 & 0 & 0 & 0 & 0 & 0 & 0 & 0 & 0 & 0 & 0 \\
                     0 & 0 & 0 & 0 & -1 & 0 & 0 & 0 & 0 & 0 & 0 & 0 & 0 & 0 \\
                     0 & 0 & 0 & 0 & 0 & -1 & 0 & 0 & 0 & 0 & 0 & 0 & 0 & 0 \\
                     0 & 0 & 0 & 0 & 0 & 0 & -1 & 0 & 0 & 0 & 0 & 0 & 0 & 0 \\
                     0 & 0 & 0 & 0 & 0 & 0 & 0 & -1 & 0 & 0 & 0 & 0 & 0 & 0 \\
                     0 & 0 & 0 & 0 & 0 & 0 & 0 & 0 & -1 & 0 & 0 & 0 & 0 & 0 \\
                     0 & 0 & 0 & 0 & 0 & 0 & 0 & 0 & 0 & -1 & 0 & 0 & 0 & 0 \\
                     0 & 0 & 0 & 0 & 0 & 0 & 0 & 0 & 0 & 0 & -1 & 0 & 0 & 0 \\
                     0 & 0 & 0 & 0 & 0 & 0 & 0 & 0 & 0 & 0 & 0 & -1 & 0 & 0 \\
                     0 & 0 & 0 & 0 & 0 & 0 & 0 & 0 & 0 & 0 & 0 & 0 & -1 & 0 \\
                    \end{array}
                    \right)}
        $$
        We finish by cyclically permuting columns $(4) - (2n+1)$, which is achieved by an even number of column transpositions which preserves the determinant:
        $$
        B\ppp_n := \tiny{\left(
                    \begin{array}{cccc|cccccccccc}
                     2 t-2 & 1 & 1-t & t^6 & 0 & 0 & 0 & 0 & 0 & 0 & 0 & 0 & 0 & 0 \\
                     -t & t-1 & 0 & 0 & 0 & 0 & 0 & 0 & 0 & 0 & 0 & 0 & 0 & 0 \\
                     1-t & 0 & 0 & -1 & 0 & 0 & 0 & 0 & 0 & 0 & 0 & 0 & 0 & 0 \\
                     -1 & 0 & t^6 & 0 & 0 & 0 & 0 & 0 & 0 & 0 & 0 & 0 & 0 & 0 \\
                     \hline
                     0 & 0 & 0 & 0 & -1 & 0 & 0 & 0 & 0 & 0 & 0 & 0 & 0 & 0 \\
                     0 & 0 & 0 & 0 & 0 & -1 & 0 & 0 & 0 & 0 & 0 & 0 & 0 & 0 \\
                     0 & 0 & 0 & 0 & 0 & 0 & -1 & 0 & 0 & 0 & 0 & 0 & 0 & 0 \\
                     0 & 0 & 0 & 0 & 0 & 0 & 0 & -1 & 0 & 0 & 0 & 0 & 0 & 0 \\
                     0 & 0 & 0 & 0 & 0 & 0 & 0 & 0 & -1 & 0 & 0 & 0 & 0 & 0 \\
                     0 & 0 & 0 & 0 & 0 & 0 & 0 & 0 & 0 & -1 & 0 & 0 & 0 & 0 \\
                     0 & 0 & 0 & 0 & 0 & 0 & 0 & 0 & 0 & 0 & -1 & 0 & 0 & 0 \\
                     0 & 0 & 0 & 0 & 0 & 0 & 0 & 0 & 0 & 0 & 0 & -1 & 0 & 0 \\
                     0 & 0 & 0 & 0 & 0 & 0 & 0 & 0 & 0 & 0 & 0 & 0 & -1 & 0 \\
                     0 & 0 & 0 & 0 & 0 & 0 & 0 & 0 & 0 & 0 & 0 & 0 & 0 & -1 \\
                    \end{array}
                    \right)}
        $$
        Notice we have shifted the horizontal and vertical lines. We now compute $\Delta_n(t)$. 
        \begin{align*}
            \Delta_n(t) &= \frac{1}{t^{n+1}}\det(B_n) = \frac{1}{t^{n+1}}\det(B\ppp_n) 
                         = \frac{1}{t^{n+1}}\det\left(
                                \tiny{  \begin{pmatrix}
                                            2t-2 & 1   & 1-t & t^n \\
                                            -t   & t-1 & 0   & 0   \\
                                            1-t  & 0   & 0   & -1  \\
                                            -1   & 0   & t^n & 0
                                        \end{pmatrix}} \oplus -\one\sub{2n-2}
                                            \right)\\
             &= -\frac{1}{t^{n+1}} + \frac{2}{t^n} - \frac{1}{t^{n-1}} + \frac{2}{t} - 3 + 2t - t^{n-1} + 2t^n - t^{n+1}
        \end{align*}
        If we had started with $A\sub{-n}$, then we would obtain the same 4-by-4 matrix but with row (4) and column (4) negated and transposed. That would not change the determinant of this particular matrix so $\Delta_n(t) = \Delta\sub{-n}(t)$. The reader may check $\Delta_n(1) = 1$. Next, we show $\Delta_n(t) < 0$ for any $t \in \sphere1 \subset \C$ satisfying $t^{\abs{n}} = -1$ and $\Re(t) < 7/8$. Assuming the hypotheses on $t$ we compute:
        \begin{align*}
            \Delta_n(t) &= -\frac{1}{t^{n+1}} + \frac{2}{t^{n}} - \frac{1}{t^{n-1}} + \frac{2}{t} - 3 + 2t - t^{n-1} + 2t^{n} - t^{n+1} \\
                        &= \frac{1}{t^n}\left( -\frac{1}{t} + 2 - t \right) + \frac{2}{t} - 3 + 2t + t^n \left(-\frac{1}{t} + 2 - t\right) \\
                        &= -(2-2\Re(t)) + (4\Re(t)-3) - (2 - 2\Re(t)) = \boxed{8 \Re(t) - 7}\\
            \Re(t) < 7/8 &\implies \Delta_n(t) < 0
        \end{align*}
        
        Lastly, we check that there is a root of unity $t_* \in \sphere1 \subset \C$ such that $\sigma(P_n,t_*) \ne 0$ for each $n \ne 0$. Pick any such $n$, it follows that there is an $\abs{n}^{th}$-root of $-1$ with real part less than $7/8$ and thus $\Delta_{n}(t)$ attains both positive and negative values on $\sphere1 \subset \C$. The intermediate value theorem guarantees there is some $t$-value at which $\Delta_n$ changes sign. \autoref{prop:alexLT} guarantees that $\sigma(P_n,t)$ must have different values on either side of this value and so $\sigma(P_n,t) \ne 0$ on at least one side. Since roots of unity are dense in $\sphere1$, we can pick some root of unity $t_*$ arbitrarily close to the point at which the signature changes sign on the side where $\sigma(P_{n},t)$ is non-zero.
        \end{proof} 
        
        Combining \autoref{prop:alexLT} with the comment after the proof of \autoref{prop:oddcase} completes the proof of \autoref{thm:main}.

\newpage
        \thispagestyle{empty}
        \newgeometry{left=5mm,top=5mm,right=5mm,bottom=5mm}
        \subsection*{The Calculus of Dualizable Patterns: a Cheat Sheet}\hfill [Charles Stine, 2022]
        \label{cheatsheet}
        \begin{multicols}{2}
        \noindent Let $K\subset\sphere3$. We adopt the convention that $E_K := \sphere3 \setminus \nbhd K$ is oriented INF at its boundary, while $\DS21 (\diffeo \nbhd K)$ is oriented ONF at its boundary. \smallskip
        
        \noindent Definitions of patterns and dual patterns: 
        \begin{gather*}
            P \subset V\ (\diffeo \DS21), \hspace{5mm} V_P := V \setminus \nbhd P \\
            V_P\ \text{is ort'd INF at $\boundary\nbhd{P}$, ONF at $\boundary V$}\\
            P\dual \subset V\dual\ (\diffeo \DS21), \qquad V\dual_{P\dual} := V\dual \setminus \nbhd{P\dual}\\
            * : V_P \to V\dual_{P\dual} \text{ (preserving orientation)} \\
            (\mu_P, \lambda_P, \mu_V, \lambda_V) \buildrel{*}\over\longmapsto (-\mu_{V\dual}, \lambda_{V\dual}, -\mu_{P\dual}, \lambda_{P\dual}) \\
        \end{gather*}
        Ways to modify patterns:
        \begin{itemize}[leftmargin=8mm, itemsep=1mm]
            \item $P\dual$ --- the dual pattern
            \item $P_n$ --- add $n$ meridonal twists
            \item $\bar P$ --- reverse all crossings and the orientation of $P$
            \item $P^m$ --- self-compose $m$ times
            \item $P\invs$ --- $(:=\bar{P}\dual)$, the concordance inverse
            \item $P\pound$ --- connect-sum pattern, $\wrap(P\pound) = 1$
        \end{itemize}
        Composition conventions:
        \begin{itemize}[leftmargin=8mm,itemsep=1mm]
            \item $P\dual_n = (P\dual)_n$ not $(P_n)\dual$
            \item $\bar{P}_n = (\bar{P})_n$ not $\bar{(P_n)}$
            \item $P_n^m = (P^m)_n$ not $(P_n)^m$
            \item $P\dual\pound = (P\dual)\pound$ not $(P\pound)\dual$
            \item $\bar{P}\pound = (\bar{P})\pound$ not $\bar{(P\pound)}$
            \item $P\pound^m = (P^m)\pound$ not $(P\pound)^m$ 
        \end{itemize}
        These conventions are chosen so that if a pattern is written with a sequence of superscripts and a sequence of subscripts, then the operations should be applied in the order: superscripts from left to right followed by subscripts from left to right (bar counts as the leftmost superscript!).\medskip
        
        \noindent Basic Identities ($P,Q$ dualizable patterns, $J,K$ knots in $\sphere3$, and $m,n \in \Z$):
        \begin{gather*}
            (P\dual)\dual = P, \quad \bar{(\bar P)} = P, \quad P_0 = P \\ 
            (P_n)\dual = P\dual_{-n}, \quad \bar{(P_n)} = \bar{P}_{-n}, \quad \bar{(P\dual)} = \bar{P}\dual \\
            (P \of P\invs)(K) \sim K \sim (P\invs \of P)(K) \\
            (P \of Q)\dual = Q\dual \of P\dual \\
            (P \of Q)_n = (P_n \of Q_n) \\
            (P_n)_m = P_{(n+m)} \\
            K\pound(J) = K \connsum J = J \connsum K = J\pound(K) \\ 
            P(K) = P(K\pound(U)) = (P \of K\pound)(U) = P \of K\pound
        \end{gather*}
        \noindent For patterns satisfying $\wrap(P)=1$:
        \begin{gather*}
            P\dual = P, \quad P\invs = \bar{P}, \quad P_n = P\\
            P = P(U)\pound, \quad P(K) = P(U) \connsum K
        \end{gather*}
        \noindent Retracing Theorems ($P$ dualizable):
        \begin{gather*}
            X_0(P) \diffeo X_0(P\dual), \quad X_n(P) \diffeo X_n(P\dual_{n}) \\
            X_n(P \connsum \bar P\dual) \diffeo X_n(\bar P_n \of P)\\
            \implies (\bar P_n \of P)\ \text{is $n$-shake-slice!}
        \end{gather*}
        \noindent Let $\sI:\mathcal C \to \mathcal A$ be a homomorphism from the smooth or TOP concordance group to an Abelian group $\mathcal A$ (usually $\Z$). If $\sI$ is 0-trace invariant, meaning, $$X_0(K)\diffeo X_0(J) \implies \sI(K) = \sI(J)$$ then, for any dualizable $P \subset V$, $K \subset \sphere3$, and $n \in \Z$, $$ \sI(P_n(K)) = \sI(K) + \sI(P_n) $$ We call such invariants \emph{compatible} with dualizable patterns.
        \end{multicols}
        \vspace{-5mm}
        \begin{figure}[!ht]
            \labellist
            \small\hair 2pt
             \pinlabel \green{$\lambda_V$} [ ] at 14 248
             \pinlabel \red{$\mu_V$} [ ] at 40 240
             \pinlabel {$m_1$} [l] at 90 211
             \pinlabel \blue{$K_0$} [ ] at 94 245
             \pinlabel \orange{$b_1$} [l] at 136 273
             \pinlabel \blue{$(+2)$} [l] at 296 267
             \pinlabel \green{$(-2)$} [ ] at 17 107
             \pinlabel \purple{Gluck twisting} [t] at 310 16
             \pinlabel \green{\tiny{$+2$}} [ ] at 348 98
             \pinlabel \green{$(0)$} [ ] at 332 134
             \pinlabel \blue{$(0)$} [ ] at 396 111
             \pinlabel \green{$P\dual$} [l] at 433 124
             \pinlabel \red{$\mu_{V\dual}$} [l] at 393 87
             \pinlabel {$P$} [ ] at 47 158
            \endlabellist
            \centering
            \includegraphics{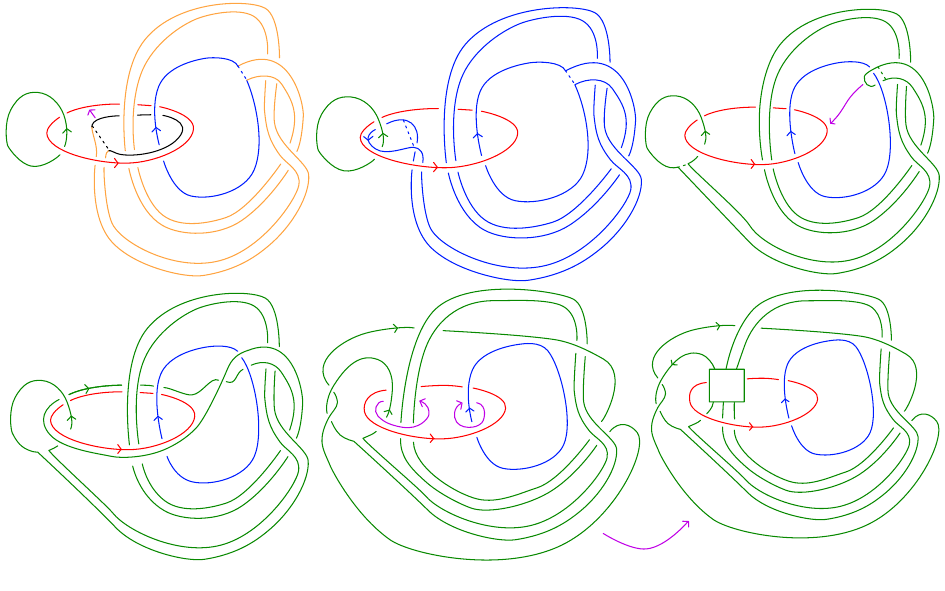}
            \caption{Working out $P\dual$ from a base-knot-with-bands description of $P$. A number in parentheses indicates a new framing after a handle slide or diffeomorphism.}
            \label{fig:cheatsheet}
        \end{figure}
        \newpage
        \restoregeometry
\bibliographystyle{alpha}
\bibliography{references}
\end{document}